\documentclass[review,leqno]{siamart1116}

\graphicspath{{fig/}}

\usepackage{graphics,graphicx,epstopdf}
	\usepackage{multirow}
\usepackage{longtable}
\usepackage[tight]{subfigure}
\usepackage{float}
\usepackage{makecell}
\usepackage{comment}
\usepackage{mathtools}
\usepackage{algorithmic}
\usepackage{algorithm}
\usepackage{tablefootnote}
\usepackage[leqno]{amsmath}
\usepackage{amsxtra, amsfonts,amscd,amssymb,graphicx,epsf}
\usepackage{essay-def-siam}
\usepackage[algo2e, vlined,ruled]{algorithm2e}

\numberwithin{equation}{section}
\numberwithin{theorem}{section}

\definecolor{am}{RGB}{0,101,189}

\newcommand{\R}{\mathbb{R}}

\newcommand{\be}{\begin{equation}}
\newcommand{\ee}{\end{equation}}
\newcommand{\bee}{\begin{equation*}}
\newcommand{\eee}{\end{equation*}}
\newcommand{\bea}{\begin{eqnarray}}
\newcommand{\eea}{\end{eqnarray}}
\newcommand{\beaa}{\begin{eqnarray*}}
\newcommand{\eeaa}{\end{eqnarray*}}

\newcommand{\Mcal}{\mathcal{M}}

\usepackage{xcolor}
\usepackage[normalem]{ulem}

\begin{document}

\title{Oracle complexities of augmented Lagrangian methods for nonsmooth manifold optimization}

\author{
Kangkang Deng\thanks{Department of Mathematics,  National University of Defense Technology, Changsha, 410073,
CHINA (\email{freedeng1208@gmail.com}).} \and 
Jiang Hu\thanks{Corresonding author. Massachusetts General Hospital and Harvard Medical School, Harvard University, Boston, MA
02114, US (\email{hujiangopt@gmail.com}).} \and
Jiayuan Wu\thanks{College of Engineering, Peking University, Beijing 100871, CHINA (\email{1901110043@pku.edu.cn}).} \and
Zaiwen Wen\thanks{Beijing International Center for Mathematical Research, Center for Machine Learning Research and College of Engineering, Peking University, Beijing 100871, CHINA (\email{wenzw@pku.edu.cn}).}
}

\maketitle

\begin{abstract}
  In this paper, we present two novel manifold inexact augmented Lagrangian methods, \textbf{ManIAL} for deterministic settings and \textbf{StoManIAL} for stochastic settings, solving nonsmooth manifold optimization problems. By using the Riemannian gradient method as a subroutine, we establish an $\mathcal{O}(\epsilon^{-3})$ oracle complexity result of \textbf{ManIAL}, matching the best-known complexity result.  Our algorithm relies on the careful selection of penalty parameters and the precise control of termination criteria for subproblems. Moreover, for cases where the smooth term follows an expectation form, our proposed \textbf{StoManIAL} utilizes a Riemannian recursive momentum method as a subroutine, and achieves an oracle complexity of $\tilde{\mathcal{O}}(\epsilon^{-3.5})$, which surpasses the best-known $\mathcal{O}(\epsilon^{-4})$ result. Numerical experiments conducted on sparse principal component analysis and sparse canonical correlation analysis demonstrate that our proposed methods outperform an existing method with the previously best-known complexity result. To the best of our knowledge, these are the first complexity results of the augmented Lagrangian methods for solving nonsmooth manifold optimization problems.
\end{abstract}

\section{Introduction}
We are concerned with the following nonsmooth composite manifold optimization problem
\begin{equation}\label{prob}
    \min_{x\in\mathcal{M}} \varphi(x):= f(x) + h(\mathcal{A}x), 
\end{equation}
where $f:\mathbb{R}^n \rightarrow \mathbb{R}$ is a continuously differentiable function, $\mathcal{A}:\mathbb{R}^n \rightarrow \mathbb{R}^m$ is a linear mapping, $h:\R^m \rightarrow (-\infty,+\infty]$ is a proper closed convex function, and $\Mcal$ is an embedded Riemannian submanifold. Hence, the objective function of \eqref{prob} can be nonconvex and nonsmooth. 
Manifold optimization with smooth objective functions (i.e., smooth manifold optimization) has been extensively studied in the literature; see \cite{AbsMahSep2008,boumal2023intromanifolds,sato2021riemannian,hu2020brief} for a comprehensive review. Recently, manifold optimization with nonsmooth objective functions has been growing in popularity due to its wide applications in sparse principal component analysis \cite{jolliffe2003modified}, nonnegative principal component analysis \cite{zass2006nonnegative,jiang2023exact}, semidefinite programming \cite{burer2003nonlinear,wang2023decomposition}, etc. 
Different from smooth manifold optimization, solving problem \eqref{prob} is challenging because of the nonsmoothness.
Recently, a range of algorithms is developed to address nonsmooth optimization problems. Notably, one category of algorithms is the augmented Lagrangian framework, as introduced by \cite{peng2022riemannian,zhou2022semismooth}. Compared with other existing works, such algorithm framework can deal with the case of the general $\mathcal{A}$ along with possible extra constraints \cite{zhou2022semismooth}. However, they both only provide an asymptotic convergence result, and the iteration complexity or oracle complexity is unclear. This naturally raises the following question:

{\textit{Can we design an augmented Lagrangian algorithm framework for \eqref{prob} with oracle complexity guarantee?}}

We answer this question affirmatively, giving two novel augmented Lagrangian algorithms that achieves the oracle complexity of $\mathcal{O}(\epsilon^{-3})$ in the deterministic setting and $\tilde{\mathcal{O}}(\epsilon^{-3.5})$ in the stochastic setting. The main contributions of this paper are given as follows:




\begin{itemize}
   \item  Developing an inexact augmented Lagrangian algorithm framework for solving problem \eqref{prob}. Our algorithm relies on carefully selecting penalty parameters and controlling termination criteria for subproblems.  In particular, when the Riemannian gradient descent method is used for the inner iterates, we prove that our algorithm, \textbf{ManIAL}, finds an $\epsilon$-stationary point with $\mathcal{O}(\epsilon^{-3})$ calls to the first-order oracle (refer to Theorem \ref{theo:oracle}). To the best of our knowledge, this is the first complexity result for the augmented Lagrangian method for solving \eqref{prob}. 
    \item  Addressing scenarios where the objective function $f$ takes the form $f(x): = \mathbb{E}_{\xi}[f(x;\xi)]$. We introduce a stochastic augmented Lagrangian method. 
    Owing to the smoothness of the subproblem, we are able to apply the Riemannian recursive momentum method as a subroutine within our framework. Consequently, we establish that \textbf{StoManIAL} achieves the oracle complexity of $\tilde{\mathcal{O}}(\epsilon^{-3.5})$ (refer to Theorem \ref{theo:sto:oracle}), which is better than the best-known $\mathcal{O}(\epsilon^{-4})$ result \cite{li2021weakly}.  
    \item 
    Our algorithm framework is highly flexible, allowing for the application of various Riemannian (stochastic) optimization methods, including first-order methods, semi-smooth Newton methods, and others, to solve sub-problems. This is due to the smoothness properties of the subproblems. Additionally, we provide two options for selecting the stopping criteria for subproblems, based on either a given residual or a given number of iterations, further enhancing the flexibility of the algorithm.
\end{itemize}

\subsection{Related works}

When the objective function is geodesically convex over a Riemannian manifold, the subgradient-type methods are studied in \cite{Bento2017Iteration,ferreira2019iteration,ferreira1998subgradient}. An asymptotic convergence result is first established in \cite{ferreira1998subgradient}, while an iteration complexity of $\mathcal{O}(\epsilon^{-2})$ is established in \cite{Bento2017Iteration,ferreira2019iteration}. For the general nonsmooth problem on the Riemannian manifold, the Riemannian subgradient may not exist. By sampling points within a neighborhood of the current iteration point at which the objective function $f$ is differentiable,
the Riemannian gradient sampling methods are developed in \cite{hosseini2018line,hosseini2017riemannian}.

 There are many works addressing nonsmooth problems in the form of \eqref{prob}.   
  When $\mathcal{A}$ in \eqref{prob} is the identity operator and the proximal operator of $h$ can be calculated analytically, Riemannian proximal gradient methods \cite{chen2020proximal,huang2022riemannian,huang2023inexact} are proposed. These methods achieve an outer iteration complexity of $\mathcal{O}(\epsilon^{-2})$ for attaining an $\epsilon$-stationary solution. It should be emphasized that each iteration in these algorithms involves solving a subproblem that lacks an explicit solution, and they utilize the semismooth Newton method to solve it.  
For the general $\mathcal{A}$, Peng et al. \cite{peng2022riemannian} develop the Riemannian smoothing gradient method by utilizing a smoothing technique from the Moreau envelope.  They show that the algorithm achieves an iteration complexity of $\mathcal{O}(\epsilon^{-3})$. Moreover, Beck and Rosset \cite{beck2023} propose a dynamic smoothing gradient descent on manifold and obtain an iteration complexity of $\mathcal{O}(\epsilon^{-3})$. Very recently, based on a similar smoothing technique, a Riemannian alternating direction method of multipliers (RADMM) is proposed in \cite{li2022riemannian}  with an iteration complexity result of $\mathcal{O}(\epsilon^{-4})$ for driving a Karush–Kuhn–Tucker (KKT) residual based stationarity. This complexity result is a bit worse than $\mathcal{O}(\epsilon^{-3})$ \cite{peng2022riemannian} because there is a lack of adaptive adjustment of the smoothing parameter. There is also a line of research based on the augmented Lagrangian function. Deng and Peng \cite{deng2022manifold} proposed a manifold inexact augmented Lagrangian method for solving problem \eqref{prob}. Zhou et al. \cite{zhou2022semismooth} consider a more general problem and propose a globalized semismooth Newton method to solve its subproblem efficiently. In these two papers, they only give the asymptotic convergence result, while the iteration complexity is missing.

In the case when the manifold $\Mcal$ is specified
by equality constraints $c(x) = 0$, e.g., the Stiefel manifold. Problem \eqref{prob} can be regarded as a constrained optimization problem with a nonsmooth and nonconvex objective function. Given that $\mathcal{M}$ is often nonconvex, we only list the related works in case of that the constraint functions are nonconvex. 
 Papers \cite{li2021rate,sahin2019inexact} propose and study the iteration complexity of augmented Lagrangian methods for solving nonlinearly
constrained nonconvex composite optimization problems. The iteration complexity results they achieve for an $\epsilon$-stationary point are both $\tilde{\mathcal{O}}(\epsilon^{-3})$.  More specifically, \cite{sahin2019inexact} uses the accelerated gradient method of \cite{ghadimi2016accelerated} to obtain the
approximate stationary point. On the other hand, the authors in \cite{li2021rate} obtain such approximate stationary point by applying
an inner accelerated proximal method as in \cite{carmon2018accelerated,kong2019complexity}, whose generated subproblems are convex. It is worth mentioning that both of these papers make a strong assumption about how the feasibility of an iterate is related to its stationarity. Lin et al. \cite{lin2019inexact} propose an inexact proximal-point penalty method by solving a sequence of penalty subproblems.  Under a non-singularity condition, they show a
complexity result of $\tilde{\mathcal{O}}(\epsilon^{-3})$.

Finally, we point out that there exist several works for nonsmooth expectation problems on the Riemannian manifold, i.e., when $f(x) = \mathbb{E}[f(x,\xi)]$ in problem \eqref{prob}. In particular, Li et al. \cite{li2021weakly} focus on a class of weakly convex problems on the Stiefel manifold. They present a Riemannian stochastic subgradient method and show that it has an iteration complexity of $\mcO(\epsilon^{-4})$ for driving a natural stationarity measure below $\epsilon$. Peng et al. \cite{peng2022riemannian} propose a Riemannian stochastic smoothing method and give an iteration complexity of $\mathcal{O}(\epsilon^{-5})$ for driving a natural stationary point.  Li et al. \cite{li2024stochastic}  propose a stochastic inexact augmented Lagrangian method to solve problems involving a nonconvex composite objective and nonconvex smooth functional constraints. Under certain regularity conditions, they establish an oracle complexity
result of $\mathcal{O}(\epsilon^{-5})$ with a stochastic first-order oracle. Wang et al. \cite{wang2022riemannian} propose two Riemannian stochastic proximal gradient methods for minimizing a nonsmooth function over the Stiefel manifold. They show that the proposed algorithm finds an $\epsilon$-stationary point with $\mathcal{O}(\epsilon^{-3})$ first-order oracle complexity. It should be emphasized that each oracle call involves a subroutine whose iteration complexity remains unclear.

In Table \ref{tab1}, we summarize our complexity results and several existing methods to produce
an $\epsilon$-stationary point. The definitions of first-order oracle in determined and stochastic settings are provided in Definition  \ref{def:detern} and \ref{def:stochastic}, respectively. The notation $\tilde{\mathcal{O}}$ denotes the dependence of a $\log$ factor. Here, we only consider the scenario where both the objective function and the nonlinear constraint are nonconvex. For the iteration complexity of ALM in convex case, please refer to \cite{aybat2012first,lu2023iteration,xu2021iteration,xu2021first,lan2016iteration}, etc.  We do not add the proximal type algorithms in \cite{chen2020proximal,wang2022riemannian,huang2022riemannian,huang2023inexact} since those algorithms involve using the semismooth Newton method to solve the subproblem, where the overall oracle complexity is not clear.   It can easily be shown that our algorithms achieve better oracle complexity results both in determined and stochastic settings. In particular, in the stochastic setting, our algorithm establishes an oracle complexity of $\tilde{\mathcal{O}}(\epsilon^{-3.5})$, which is better than the best-known $\mathcal{O}(\epsilon^{-4})$ result. 

\begin{table} 
   \caption{Comparison of the oracle complexity results of several methods in the literature to our method to produce an
$\epsilon$-stationary point}
    \centering
    \begin{tabular}{|c|c|c|c|c|}
  \hline
  Algorithms & $\mathcal{A}$ & Constraints & Objective    & Complexity \\ \hline
  iALM \cite{sahin2019inexact} & $\mathcal{I}$ & $c(x) = 0$ & determine   &    $\mathcal{O}(\epsilon^{-4})$\\ \hline
  iPPP \cite{lin2019inexact} & $\mathcal{I}$ & $c(x) = 0$ & determine &    $\tilde{\mathcal{O}}(\epsilon^{-3})$\tablefootnote{We only give the case of $c(x)$ is nonconvex and a non-singularity condition holds
on $c(x)$}\\ \hline
 improved iALM \cite{li2021rate} & $\mathcal{I}$ & $c(x) = 0$ & determine &    $\tilde{\mathcal{O}}(\epsilon^{-3})$\\ \hline
  Sto-iALM \cite{li2024stochastic} & $\mathcal{I}$ & $\mathbb{E}_{\xi}[c(x,\xi)] = 0$ & stochastic &    $\mathcal{O}(\epsilon^{-5})$\\ \hline
   RADMM \cite{li2022riemannian} & general & compact submanifold & determine     & $\mathcal{O}(\epsilon^{-4})$ \\ \hline
   DSGM \cite{beck2023} & general & compact submanifold & determine     & $\mathcal{O}(\epsilon^{-3})$ \\ \hline
  \multirow{2}{*}{subgradient \cite{li2021weakly}} & \multirow{2}{*}{general}  &  \multirow{2}{*}{Stiefel manifold} & determine    &  $\mathcal{O}(\epsilon^{-4})$ \\ 
     & &  & stochastic      &  $\mathcal{O}(\epsilon^{-4})$ \\ \hline
      \multirow{2}{*}{smoothing \cite{peng2022riemannian}} & \multirow{2}{*}{general} &  \multirow{2}{*}{compact submanifold} & determine     & $\mathcal{O}(\epsilon^{-3})$ \\
 &  &  & stochastic    &  $\mathcal{O}(\epsilon^{-5})$ \\ \hline
  \multirow{2}{*}{ \textbf{this paper}} & \multirow{2}{*}{general} &  \multirow{2}{*}{compact submanifold} & determine  &   $\mathcal{O}(\epsilon^{-3})$\\
 & & & stochastic &     $\tilde{\mathcal{O}}(\epsilon^{-3.5})$\\
  \hline
\end{tabular}
 \label{tab1}
\end{table}

\subsection{Notation} Let $\left<\cdot,\cdot\right>$ and $\|\cdot \|$ be the Euclidean product and induced Euclidean norm. Given a matrix $A$, we use $\|A\|_F$ to denote the Frobenius norm, $\|A\|_1:=\sum_{ij}\vert A_{ij}\vert$ to denote the $\ell_1$ norm. For a vector $x$, we use $\|x\|_2$ and $\|x\|_1$ to denote its Euclidean norm and $\ell_1$ norm, respectively. The indicator function of a set $\mathcal{C}$, denoted by $\delta_{\mathcal{C}}$, is set to  0 on $\mathcal{C}$ and $+\infty$ otherwise. The distance from $x$ to $\mathcal{C}$ is denoted by $\mathrm{dist}(x,\mathcal{C}): = \min_{y\in\mathcal{C}}\|x-y\|$. We use $\nabla f(x)$ and $\grad f(x)$ to denote the Euclidean gradient and Riemannian gradient of $f$, respectively.

  \section{Preliminaries}\label{sec-preli}

In this section, we introduce some necessary concepts for Riemannian optimization and the Moreau envelope.
\subsection{Riemannian optimization}
 An $n$-dimensional smooth manifold $\mathcal{M}$ is an   $n$-dimensional topological manifold equipped with a smooth structure, where each point has a neighborhood that is diffeomorphism to the $n$-dimensional Euclidean space. For all $x\in\mathcal{M}$, there exists a chart $(U,\psi)$ such that $U$ is an open set and $\psi$ is a diffeomorphism between $U$ and an open set $\psi(U)$ in the Euclidean space.   A tangent vector $\eta_x$ to $\mathcal{M}$ at $x$ is defined as tangents of parametrized curves $\gamma$ on $\mathcal{M}$ such that $\gamma(0) = x$ and
\[\nonumber
  \eta_x u : = \dot{\gamma}(0)u = \left.\frac{d(u(\gamma(t)))}{dt} \right\vert_{t=0} , \forall u\in  \wp_x\mathcal{M},
\]
where $\wp_x\mathcal{M}$ is the set of all real-valued functions $f$ defined in a neighborhood of $x$ in $ \mathcal{M}$. Then, the tangent space $T_x\mathcal{M}$ of a manifold $\mathcal{M}$ at $x$ is defined as the set of all tangent vectors at point $x$.  The manifold $\mathcal{M}$ is called a Riemannian manifold if it is equipped with an inner product on the tangent space $T_x\mathcal{M}$ at each $x\in\mathcal{M}$. If $\mathcal{M}$ is a Riemannian submanifold of an Euclidean space $\mathcal{E}$, the inner product is defined as the Euclidean inner product: $\left<\eta_x,\xi_x\right> = \mathrm{tr}(\eta_x^\top \xi_x)$. The Riemannian gradient $\grad  f(x) \in T_x\mathcal{M}$ is the unique tangent vector satisfying
$$  \left< \grad f(x), \xi \right> = df(x)[\xi], \forall \xi\in T_x\mathcal{M}. $$ If $\mathcal{M}$ is a compact Riemannian manifold embedded in an Euclidean space, we have that $\grad f(x) = \mathcal{P}_{T_x\mathcal{M}}(\nabla f(x))$, where $\nabla f(x)$ is the Euclidean gradient, $\mathcal{P}_{T_x \mathcal{M}}$ is the projection operator onto the tangent space $T_x \mathcal{M}$. 
 The retraction operator is one of the most important ingredients for manifold optimization, which turns an element of $T_x\mathcal{M}$ into a point in $\mathcal{M}$.


\begin{definition}[Retraction, \cite{AbsMahSep2008}]\label{def-retr}
  A retraction on a manifold $\mathcal{M}$ is a smooth mapping $\mathcal{R}:T\mathcal{M}\rightarrow \mathcal{M}$ with the following properties. Let $\mathcal{R}_x:T_x\mathcal{M} \rightarrow \mathcal{M}$ be the restriction of $\mathcal{R}$ at $x$. It satisfies
\begin{itemize}
  \item $\mathcal{R}_x(0_x) = x$, where $0_x$ is the zero element of $T_x\mathcal{M}$,
  \item $d\mathcal{R}_x(0_x) = id_{T_x\mathcal{M}}$,where $id_{T_x\mathcal{M}}$ is the identity mapping on $T_x\mathcal{M}$.
\end{itemize}
\end{definition}
We have the following Lipschitz-type inequalities on the retraction on the compact submanifold.
\begin{proposition}[{\cite[Appendix B]{grocf}}]
    Let $\mathcal{R}$ be a retraction operator on a compact submanifold $\Mcal$. Then, there exist two positive constants $\alpha, \beta$ such that
    for all $x\in \mathcal{M}$ and  all $u \in T_{x}\mathcal{M}$, we have
   \begin{equation}\label{eq:retrac-lipscitz}
  \left\{ \begin{aligned}
     \|\mathcal{R}_x(u) - x\| &\leq \alpha \|u\|, \\
     \|\mathcal{R}_x(u) - x - u\| &\leq \beta \|u\|^2.
   \end{aligned}\right.
   \end{equation}
\end{proposition}

%
Another basic ingredient for manifold optimization is the vector transport.
\begin{definition}[Vector Transport, \cite{AbsMahSep2008}]
 The vector transport $\mathcal{T}$ is a smooth mapping
 \begin{equation}
   T\mathcal{M}\oplus T\mathcal{M} \rightarrow T\mathcal{M}:(\eta_x,\xi_x)\mapsto \mathcal{T}_{\eta_x}(\xi_x)\in T\mathcal{M}
 \end{equation}
 satisfying the following properties for all $x\in\mathcal{M}$:
 \begin{itemize}
   \item for any $\xi_x\in T_x\mathcal{M}$, $\mathcal{T}_{0_x}\xi_x = \xi_x$,
   \item $\mathcal{T}_{\eta_x}(a\xi_x + b\gamma_x)  = a \mathcal{T}_{\eta_x}(\xi_x) + b\mathcal{T}_{\eta_x}(\gamma_x)$.
 \end{itemize}
 When there exists $\mathcal{R}$ such that $y=\mathcal{R}_{x}(\eta_x)$, we write $\mathcal{T}_{x}^y(\xi_x) = \mathcal{T}_{\eta_x}(\xi_x)$. 
\end{definition}
%

\subsection{Moreau envelope and retraction smoothness}

We first introduce the concept of the Moreau envelope.
\begin{definition}\label{def:moreau}
For a proper, convex and lower semicontinuous function $h:\mathbb{R}^n \rightarrow \mathbb{R}$, the Moreau envelope of $h$ with the parameter $\mu>0$ is given by
 \begin{equation*}
     M_h^{\mu}(y): = \inf_{z\in\mathbb{R}^n} \left\{h(z) + \frac{1}{2\mu }\|z- y\|^2  \right\}.
 \end{equation*}
\end{definition}

The following result demonstrates that the Moreau envelope of a convex function is not only continuously differentiable but also possesses a bounded gradient norm.
\begin{proposition}\label{propos-1}\cite[Lemma 3.3]{bohm2021variable}
Let $h$ be a proper, convex, and  $\ell_h$-Lipschitz continuous. Then,  the Moreau envelope $M_h^{\mu}$ has Lipschitz continuous gradient over $\mathbb{R}^n$ with constant $\mu^{-1}$, and the gradient is given by 
  \begin{equation}\label{moreau envelope gradient}
  \nabla M_h^{\mu}(x)  = \frac{1}{\mu}\left(x - \mathrm{prox}_{\mu h}(x) \right) \in \partial h\left(\mathrm{prox}_{\mu h}(x)\right),
  \end{equation}
where $\partial h$ denote the subdifferential of $h$,   $\mathrm{prox}_{\mu h}(x): = \arg\min_{y\in\mathbb{R}^n}
\{ h(y) + \frac{1}{2\mu}\|y-x\|^2 \}$ is the proximal operator. Moreover, it holds that
\begin{equation} \label{eq:moreau-gradient-bound}
\|\nabla  M_h^{\mu}(x)\| \leq \ell_h, \; \|x - \mathrm{prox}_{\mu h}(x) \| \leq \mu \ell_h. 
\end{equation}


\end{proposition}

Next, we provide the definition of retraction smoothness. This concept plays a crucial role in the convergence analysis of the algorithms proposed in the subsequent sections.
\begin{definition}\cite[Retraction smooth]{grocf}\label{def:retr-smooth}
  A function $f:\mathcal{M}\rightarrow\mathbb{R}$ is said to be  retraction smooth (short to retr-smooth) with constant $\ell$ and a retraction $\mathcal{R}$,  if for   $\forall~x, y\in\mathcal{M}$  it holds that
  \begin{equation}\label{eq:taylor expansion}
    f(y) \leq f(x) + \left<\grad f(x), \eta\right> + \frac{\ell}{2}\|\eta\|^2,
  \end{equation}
  where $\eta\in T_x\mathcal{M}$ and $y= \mathcal{R}_x(\eta)$.
\end{definition}

Similar to the Lipschitz continuity of gradient in Euclidean space, we give the following extended definition of the Lipschitz continuity of the Riemannian gradient.

\begin{definition}\label{def:gradient-lipsch}
    The Riemannian gradient $\grad f$ is
Lipschitz continuous with respect to a vector transport $\mathcal{T}$ if there exists a positive constant $L$ such that for all $x, \xi \in T_x{\Mcal}, y = \mathcal{R}_x(\xi)$, 
\begin{equation}\label{eq:def:grad-lipschitz}
    \|\grad f(x) - \mathcal{T}_y^x(\grad f(y)) \| \leq L\| \xi \|. 
\end{equation}
\end{definition}
The following lemma shows that the Lipschitz continuity of $\grad f(x)$ can be deduced by the Lipschitz continuity of $\nabla f(x)$.
\begin{lemma}\label{lem:eucli-rieman-lipsch}
  Suppose that $\mathcal{M}$ is a compact submanifold embedded in the Euclidean space, given 
$x,y\in \mathcal{M}$ and $u\in T_x\mathcal{M}$, the vector transport $\mathcal{T}$ is defined as $\mathcal{T}_x^y(u): = D\mathcal{R}_x[\xi](u)$, where $y = \mathcal{R}_x(\xi)$. Denote $\zeta: =\max_{x\in \text{conv}(\mathcal{M})} \| D^2 \mathcal{R}_x(\cdot)\|$ and $G: = \max_{x\in \Mcal} \|\nabla f(x) \|$. Let $L_p$ be the Lipschitz constant of $\mathcal{P}_{T_{x} \Mcal}$ over $x \in \Mcal$. For a function $f$ with Lipschitz continuous gradient with constant $L$, i.e., $\|\nabla f(x) - \nabla f(y)\| \leq L\|x - y\|$, then we have that
    \begin{equation}
        \|\grad f(x) - \mathcal{T}_y^x (\grad f(y)) \| \leq ((\alpha L_p + \zeta) G + \alpha  L) \|\xi\|. 
    \end{equation}
\end{lemma}
\begin{proof}
    It follows from the definition of Riemannian gradient and vector transport that 
     \begin{equation}
        \begin{aligned}
          & \|\grad f(y) - \mathcal{T}_x^y (\grad f(x)) \| \\
           \leq &    \|\grad f(y) - \grad f(x) \| +  \|\grad f(x) - \mathcal{T}_x^y (\grad f(x)) \| \\
            \leq & \|\grad f(y) - \mathcal{P}_{T_x \Mcal}(\nabla f(y))\| + \| \mathcal{P}_{T_x \Mcal}(\nabla f(y) - \nabla f(x)) \|  \\
            & + \| D\mathcal{R}_x(0_x)[ \grad f(x)] - D\mathcal{R}_x(\xi)[\grad f(x)] \| \\
            \leq & (L_p G + L) \| \|y-x \| +  \zeta \|\xi\| \| \grad f(x)\| \\
            \leq & ((\alpha L_p + \zeta) G + \alpha  L) \|\xi\|,
        \end{aligned}
    \end{equation}
    where  the third inequality utilizes the Lipschitz continuity of $\nabla f(x)$ and $\mathcal{P}_{T_{x} \Mcal}$ over $x \in \Mcal$, the last inequality follows from \eqref{eq:retrac-lipscitz} and $\|\grad f(x)\| \leq \|\nabla f(x)\| \leq G$.    This proof is completed.
\end{proof}

\section{Augmented Lagrangian method}
In this section, we present an augmented Lagrangian method to solve \eqref{prob}. The subproblem is solved by the Riemannian gradient descent method. We also give the iteration complexity result.
Throughout this paper, we make the following assumptions.
\begin{assumption}\label{assum}
The following assumptions hold:
\begin{enumerate}
  \item[1.]
    The manifold $\mathcal{M}$ is a compact Riemannian submanifold embedded in  $\mathbb{R}^n$;
\item[2.]
    The function $f$ is $\ell_{\nabla f}$-smooth but not necessarily convex. 
  The function $h$ is convex and $\ell_h$-Lipschitz continuous but is not necessarily smooth; $\varphi$ is bounded from below.
  
\end{enumerate}
\end{assumption}

\subsection{KKT condition and stationary point}

We say $x \in \mathcal{M}$ is an $\epsilon$-stationary point of problem \eqref{prob} if there exists $\zeta \in \partial h(\mathcal{A} x)$ such that $\left\|\mathcal{P}_{T_x \mathcal{M}}\left(\nabla f(x)+\mathcal{A}^* \zeta\right)\right\| \leq \epsilon$. As in \cite{zhang2020complexity}, there is no finite time algorithm that can guarantee $\epsilon$-stationarity in the nonconvex nonsmooth setting. Motivated by the notion of $(\delta, \epsilon)$-stationarity introduced in \cite{zhang2020complexity}, we introduce a generalized $\epsilon$-stationarity for problem \eqref{prob} based on the Lagrangian function.

Since problem \eqref{prob} contains the nonsmooth term, we introduce the auxiliary variable $y = \mathcal{A}x$, which results in
the following equivalent split problem:
\begin{equation}\label{prob:split}
    \min_{x\in\mathcal{M},y\in \mathbb{R}^m} f(x) + h(y), \; \mathrm{s.t.}\; \mathcal{A}x = y. 
\end{equation}
The Lagrangian function $l:\mathcal{M}\times \mathbb{R}^m \times \mathbb{R}^m \rightarrow \mathbb{R}$ of \eqref{prob:split} is given by:
\begin{equation}\label{def:lagrangian}
    l(x,y,z) = f(x) + h(y) - \left<z, \mathcal{A}x - y\right>,
\end{equation}
where $z$ is the corresponding Lagrangian multiplier. Then the KKT condition of \eqref{prob} is given as follows \cite{deng2022manifold}: A point $x\in \mathcal{M}$ satisfies the KKT condition if there exists $y,z\in\mathbb{R}^m$ such that
\begin{equation}\label{def:kkt-deter}
    \left\{
\begin{aligned}
    0 &=  \mathcal{P}_{T_x\mathcal{M}}\left(\nabla f(x) - \mathcal{A}^*z \right), \\
    0 & \in z+ \partial h(y), \\
   0 & =  \mathcal{A}x -y.
\end{aligned}
    \right.
\end{equation}

Based on the KKT condition, we give the definition of $\epsilon$-stationary point for \eqref{prob}:
\begin{definition}\label{def:epsilon-deter}
We say $x\in\mathcal{M}$ is an $\epsilon$-stationary point of \eqref{prob} if there exists $y,z\in\mathbb{R}^m$ such that 
\begin{equation}\label{eq:epsi-kkt}
    \left\{
\begin{aligned}
    \left\|\mathcal{P}_{T_x\mathcal{M}}\left(\nabla f(x) - \mathcal{A}^*z \right) \right\|\leq \epsilon, \\
   \mathrm{dist}( -z, \partial h(y)) \leq 
\epsilon, \\
   \| \mathcal{A}x - y\| \leq \epsilon.
\end{aligned}
    \right.
\end{equation}
In other words, $(x,y,z)$ is an $\epsilon$-KKT point pair of \eqref{prob}. 
\end{definition}


\subsection{Augmented Lagrangian function of \eqref{prob}}
Before constructing the augmented Lagrangian function of \eqref{prob}, we first focus on the equivalent split problem \eqref{prob:split}. In particular, the augmented Lagrangian function of \eqref{prob:split} is constructed as follows:
\begin{equation}\label{eq:al:split}
\begin{aligned}
    \mathbb{L}_{\sigma}(x,y;z): &= l(x,y,z) + \frac{\sigma}{2}\|\mathcal{A}x - y\|^2\\
    & =f(x) + h(y) - \left<z, \mathcal{A}x - y\right> + \frac{\sigma}{2}\|\mathcal{A}x - y\|^2,
    \end{aligned}
\end{equation}
where $\sigma$ is the penalty parameter. It is shown that the augmented Lagrangian method involves the variable $x$ and the auxiliary variables $y$. However, for the original problem \eqref{prob}, which only features the variable $x$, the corresponding augmented Lagrangian function should solely include $x$ as well as the associated multipliers. 

An intuitive way of introducing the augmented Lagrangian function associated with \eqref{prob} is to eliminate the auxiliary variables in \eqref{eq:al:split}, i.e.,
\begin{equation}\label{eq:al:prob}
    \mathcal{L}_{\sigma}(x,z): = \min_{y\in\mathbb{R}^m} \mathbb{L}_{\sigma}(x,y;z).
\end{equation}
Note that the optimal solution $y^*$ of the above minimization problem  has the following closed-form solution:
\begin{equation}\label{eq:solution-y}
    y^* = \mathrm{prox}_{h/\sigma}(\mathcal{A}x - z/\sigma).
\end{equation}
Therefore, one can obtain the explicit form of $\mathcal{L}_{\sigma}(x,z)$:
\begin{equation}
    \begin{aligned}
        \mathcal{L}_{\sigma}(x,z) =  &~~ f(x) + g(\mbox{prox}_{ h/\sigma}(\mathcal{A}(x)- z/\sigma_k)) \\
             & + \frac{\sigma}{2}\left\|\mathcal{A}(x)- z/\sigma - \mbox{prox}_{ h/\sigma}(\mathcal{A}(x)- z/\sigma)\right\|_2^2 - \frac{1}{2\sigma}\|z\|_2^2 \\
  = & ~~ f(x) + M_{h}^{1/\sigma}\left( \mathcal{A}(x)-\frac{1}{\sigma} z\right) -\frac{1}{2\sigma}\|z\|_2^2,
    \end{aligned}
\end{equation}
where $M_{h}^{1/\sigma}$ is the Moreau envelope given in Definition \ref{def:moreau}. It is important to emphasize that the augmented Lagrangian function \eqref{eq:al:prob} only retains $x$ as the primal variable. The Lagrangian multiplier $z$ corresponds to the nonsmooth function $h$. 
When comparing $\mathcal{L}_{\sigma}$ with $\mathbb{L}_{\sigma}$, it becomes evident that $\mathcal{L}_{\sigma}$ involves fewer variables. We will also demonstrate that the augmented Lagrangian function \eqref{eq:al:prob} possesses favorable properties, such as being continuously differentiable. This characteristic proves to be advantageous when it comes to designing and analyzing algorithms \cite{deng2023augmented,li2018highly}.


\subsection{Algorithm}

As shown in \cite{deng2022manifold}, a manifold augmented Lagrangian method based on the augmented Lagrangian function \eqref{eq:al:split} can be written as
\begin{equation}\label{alm:split}
    \left\{
\begin{aligned}
    (x^{k+1},y^{k+1}) & = \arg\min_{x\in\mathcal{M}} \mathbb{L}_{\sigma_k}(x,y,z^k), \\
    z^{k+1} &= z^k - \sigma_k (\mathcal{A}x^{k+1} - y^{k+1}), \\
    \sigma_{k+1} &= \sigma_k \uparrow.  
\end{aligned}
    \right.
\end{equation}
By the definition of $\mathcal{L}_{\sigma}$ and \eqref{eq:solution-y}, we know that \eqref{alm:split} can be rewritten as
\begin{equation}\label{alm:original}
    \left\{
\begin{aligned}
    x^{k+1} & = \arg\min_{x\in\mathcal{M}} \mathcal{L}_{\sigma_k}(x,z^k), \\
    y^{k+1} & = \mathrm{prox}_{h/\sigma_k}(\mathcal{A}x^{k+1} - z^k/\sigma_k), \\
    z^{k+1} &= z^k - \sigma_k (\mathcal{A}x^{k+1} - y^{k+1}), \\
    \sigma_{k+1} &= \sigma_k \uparrow.  
\end{aligned}
    \right.
\end{equation}
We call \eqref{alm:original}  the augmented Lagrangian method based on \eqref{eq:al:prob} for \eqref{prob}.  The authors in \cite{deng2022manifold} show that an inexact version of \eqref{alm:original} converges to a stationary point. However, the iteration complexity result is unknown.  Motivated by \cite{sahin2019inexact}, we design a novel manifold inexact augmented Lagrangian framework, as outlined in Algorithm \ref{alg-manial} to efficiently find an $\epsilon$-KKT point of \eqref{prob}. 
The key distinction from \eqref{alm:original} lies in the update strategy for the Lagrangian multiplier $z^k$ and the careful selection of $\sigma_k$. This strategic choice will enable us to bound the norm of $z^k$ and ensures that $\frac{\|z^k\|}{\sigma_k}$ converges to zero. 

\begin{algorithm}[htb]
   \caption{(\textbf{ManIAL})}
   \label{alg-manial}
\begin{algorithmic}[1]
   \STATE {\bfseries Given: }   increasing positive sequence $\{\sigma_k\}_{k\geq 1}$, decreasing positive sequence $\{\epsilon_k\}_{k\geq 1}$.
   \STATE {\bfseries Initialize:  } Primal variable $x^0$, dual variable $z^0 = 0$, and dual step size $\beta_0$. Denote $y^0 = \mathcal{A} x^0$.  

  \FOR{$k=0,1,2,\cdots$}
   \STATE Apply subroutine to solve the following subproblem
   \be\label{sub-x}
 \min_{x\in \Mcal} \mathcal{L}_{\sigma_k}(x,z^k)
   \ee
   and return $x^{k+1}$ with two options:
   \begin{itemize}
       \item \textbf{Option I}: $x^{k+1}$ satisfies 
        \begin{equation}\label{eq:solve-x}
   \|\grad  \mathcal{L}_{\sigma_k}(x^{k+1},z^k)\| \leq \epsilon_k.
\end{equation}
\item \textbf{Option II}: return $x^{k+1}$ with the fixed iteration number $2^k$. 
   \end{itemize}

\STATE Update the auxiliary variable:
$$
    y^{k+1} = \mathrm{prox}_{h/\sigma_k}(\mathcal{A}x^{k+1} - z^k/\sigma_k).
$$
   \STATE Update the dual step size:
   \begin{equation}\label{iter:beta}
       \beta_{k+1} = \beta_0\min\left(\frac{\|\mathcal{A}x^0 - y^0\|\log^22}{\|\mathcal{A}x^{k+1}-y^{k+1}\|(k+1)^2\log(k+2)},1\right)
   \end{equation}
   \STATE Update the dual variable: \begin{equation}\label{iter:z}
       z^{k+1} = z^k - \beta_{k+1}(\mathcal{A}x^{k+1} - y^{k+1}).
   \end{equation}
   \ENDFOR
\end{algorithmic}
\end{algorithm}


In Algorithm \ref{alg-manial}, we introduce the dual stepsize $\beta_k$, which is used in the update of the dual variable $z^{k+1}$ (also called the Lagrangian multiplier). The update rule \eqref{iter:beta} of $\beta_{k+1}$ is to make sure that the dual variable $z^{k+1}$ is bounded when $k\rightarrow \infty$. The detailed derivation is given in the proof of Theorem \ref{the:milam:outer}.  Moreover, for the convenience of checking the termination criteria and subsequent convergence analysis, we introduce an auxiliary sequence denoted as $y^k$.  The main difference between our algorithm and the one presented in \cite{deng2022manifold} lies in our algorithm's reliance on a more refined selection of step sizes and improved control over the accuracy of $x$-subproblem solutions. Another crucial difference is the stopping criteria for the subproblem. In particular, we provide two options based on either a given residual or a given number of iterations, further enhancing the flexibility of the algorithm  These modifications allow us to derive results pertaining to the iterative complexity of the algorithm.

   The function $\mathcal{L}_{\sigma}(x,z)$ with respect to $x$ exhibits favorable properties, such as continuous differentiability, enabling a broader range of choices in selecting subproblem-solving algorithms. Let us denote $\psi_k(x) := \mathcal{L}_{\sigma_k}(x,z^k)$.
By the property of the proximal mapping and the Moreau envelope in Proposition \ref{propos-1}, one can readily conclude that $\psi_k(x)$ is continuously differentiable and
\begin{equation}\label{eq:grad-psik}
\begin{aligned}
  \nabla \psi_k(x) & = \nabla f(x) + \sigma_k \mathcal{A}^*\left(\mathcal{A}x-\frac{1}{\sigma_k} z^k - \mbox{prox}_{ h/\sigma_k}(\mathcal{A}(x)-\frac{1}{\sigma_k} z^k)\right).
  \end{aligned}
\end{equation}  
The main computational cost of Algorithm \ref{alg-manial} lies in solving the subproblem \eqref{sub-x}. 
In the next subsection, we will give a subroutine to solve \eqref{sub-x}. 

\subsection{Subproblem solver} \label{sec-sub}
In this subsection, we give a subroutine to find each $x^{k+1}$ in Algorithm \ref{alg-manial}. We first present a general algorithm for smooth problems on Riemannian manifolds, and then try to apply the algorithm to solve the subproblem \eqref{sub-x} in Algorithm \ref{alg-manial} by checking the related condition.   Let us focus on the following smooth problem on the Riemannian manifold:
\begin{equation} \label{prob:xsub}
    \min_{x\in \Mcal } \psi(x),
\end{equation}
where $\Mcal$ is a Riemannian submanifold embedded in Euclidean space.  Denote $\psi_{\min}: = \inf_{x\in \Mcal} \psi(x)$. We assume that $\psi$ is retr-smooth with $L$ and $\mathcal{R}$, i.e.,  
\begin{equation}\label{assum:retra-smooth}
    \psi(\mathcal{R}_x(\eta)) \leq \psi(x) + \langle \grad \psi(x), \eta \rangle + \frac{L}{2}\|\eta\|^2, ~ \forall x\in \Mcal, \eta\in T_x\Mcal.
\end{equation}

\begin{algorithm}[htb]
   \caption{Riemannian gradient descent method, $\text{RGD}(\psi, L, T,x_0)$ }
   \label{alg4}
\begin{algorithmic}[1]
   \STATE {\bfseries Given: }   $x_0\in\mathcal{M}$, $L>0$, an integer $T>0$.
   \FOR{$t = 0,1,\cdots, T$}
   \STATE Set the stepsize $\alpha_t = \frac{1}{L}$.
   \STATE Compute the next iterate $x_{t+1}$:
   \begin{equation}\label{iter:xk}
  x_{t+1} = \mathcal{R}_{x_t}(-\alpha_t \grad \psi(x_t)).
\end{equation}
   \ENDFOR
   
   \textbf{Output:} $x_* = x_{t^*}$, where $t^* = \arg\min_{0\leq t\leq T} \|\grad \psi(x_t)\|$. 
\end{algorithmic}
\end{algorithm}

We adopt a Riemannian gradient descent method (RGD) \cite{AbsMahSep2008,boumal2023intromanifolds} to solve problem \eqref{prob:xsub}.  The detailed algorithm is given in Algorithm \ref{alg4}. The following lemma gives the iteration complexity result of Algorithm \ref{alg4}.
\begin{lemma}[Iteration complexity of RGD]\label{lem:mialm2:rgd}
Suppose that $\psi$ is retr-smooth with constant $L$ and a retraction operator $\mathcal{R}$.  Let  $\{x_t\}$ be the iterative sequence of Algorithm \ref{alg4}. Given an integer $T$, it holds that
  \begin{equation*}
    \| \grad \psi(x_*)\| \leq \frac{\sqrt{2L(\psi(x_0) - \psi_{\min})}}{\sqrt{T}}. 
  \end{equation*} 
\end{lemma}

\begin{proof}
    It follows from \eqref{assum:retra-smooth} and \eqref{iter:xk} that
    \begin{equation}
    \begin{aligned}
         \psi(x_{t+1}) & \leq \psi(x_t) - \alpha_t\langle \grad \psi(x_t), \grad \psi(x_t) \rangle + \frac{\alpha_t^2 L}{2}\|\grad \psi(x_t)\|^2 \\
         & \leq \psi(x_t) - \frac{1}{2L}\|\grad \psi(x_t)\|^2.
    \end{aligned}    
    \end{equation}
    By summing both sides of the above inequality over $t=0,\cdots, T-1$, we get
    \begin{equation}
        \psi(x_{T}) \leq \psi(x_0) - \frac{1}{2L}\sum_{t=0}^T \|\grad \psi(x_t)\|^2. 
    \end{equation}
    This implies that 
    \begin{equation}
     \|\grad \psi(x_*)\| =    \min_{0\leq t \leq T } \|\grad \psi(x_t)\| \leq \frac{\sqrt{2L(\psi(x_0) - \psi_{\min})}}{\sqrt{T}}.
    \end{equation}
   We complete the proof. 
\end{proof}

To apply Algorithm \ref{alg4} for computing $x^{k+1}$ in Algorithm \ref{alg-manial}, it is crucial to demonstrate that $\psi_k(x)$ is retr-smooth, meaning it satisfies \eqref{assum:retra-smooth}. The following lemma provides the essential insight.
\begin{lemma}\label{Euclidean vs manifold} 
  Suppose that Assumption \ref{assum} holds.  Then, for $\psi_k$, there exists finite $G$ such that $\|\nabla \psi_k(x)\|\leq G$ for all $x\in\mathcal{M}$, and $\psi_k$ is retr-smooth in the sense that
  \begin{equation}\label{equ:L2}
    \psi_k(\mathcal{R}_{x}(\eta)) \leq \psi_k(x) + \left<\eta,\grad \psi_k(x)\right> + \frac{L_{k}}{2}\|\eta\|^2
  \end{equation}
  for all $\eta\in T_{x}\mathcal{M}$, where $L_{k} = \alpha^2 (\ell_{\nabla f} + \sigma_k \|\mathcal{A}\|^2 )  +2 G \beta.$   
\end{lemma}
\begin{proof}
    By Proposition \ref{propos-1}, one shows that $\nabla \psi_k$  is Lipschitz continuous with constant $ \ell_{\nabla f} + \sigma_k \|\mathcal{A}\|^2 $.  Since $\nabla \psi_k$ is continuous on the compact manifold $\mathcal{M}$, it follows from \eqref{eq:moreau-gradient-bound} that there exists $G>0$ such that $\|\nabla \psi_k(x)\|\leq G$ for all $x\in\mathcal{M}$ and any $k>0$. The proof is completed by combining \eqref{eq:retrac-lipscitz} with Lemma 2.7 in \cite{grocf}. 
\end{proof}

With Lemmas \ref{lem:mialm2:rgd} and \ref{Euclidean vs manifold}, we can apply Algorithm \ref{alg4} to find each $x^{k+1}$. The lemma below gives the inner iteration
complexity for the $k$-th outer iteration of Algorithm \ref{alg-manial}.

\begin{lemma}\label{theo:rgd}
   Suppose that Assumption \ref{assum} holds.  Let $(x^k,y^k,z^k)$ be the $k$-iterate generated in Algorithm \ref{alg-manial}. We can find $x^{k+1}$ with \textbf{Option I} by the following call
    \begin{equation}
      x^{k+1} = RGD(\psi_k, L_k, T_k,x^k),
    \end{equation}
    where 
    $
    L_k = \alpha^2 (\ell_{\nabla f} + \sigma_k \|\mathcal{A}\|^2 )  +2 G \beta,~~T_k = \frac{L_k(\psi_k(x^k) - \psi_{k,\min} )}{\epsilon_k^2} 
    $, and $\psi_{k,\min} = \inf_{x\in \Mcal} \psi_k(x)$. 
    Moreover, we can also find $x^{k+1}$ with \textbf{Option II} by the following call
     \begin{equation}
      x^{k+1} = RGD(\psi_k, L_k, 2^k,x^k),
    \end{equation}
This implies that
\begin{equation}
 \| \grad \psi_k(x^{k+1}) \|\leq  \frac{\sqrt{L_k(\psi_k(x^k) - \psi_{k,\min})}}{\sqrt{2^k}}.
\end{equation}
\end{lemma}

\begin{proof}
    For \textbf{Option I}, by Lemma \ref{lem:mialm2:rgd}, we find $x^{k+1}$ that satisfy 
    \begin{equation}
        \| \grad \psi_k(x^{k+1}) \|\leq  \frac{\sqrt{L_k(\psi_k(x^k) - \psi_{k,\min})}}{T_k} \leq \epsilon_k.
    \end{equation}
    When $T_k = \frac{L_k(\psi_k(x^k) - \psi_{k,\min} )}{\epsilon_k^2} 
    $, the above inequality holds. For \textbf{Option II}, it directly follows from Lemma \ref{lem:mialm2:rgd}. 
\end{proof}

\subsection{Convergence analysis of \textbf{ManIAL}}
We have shown the oracle complexity for solving the subproblems of \textbf{ManIAL}. Now, let us first investigate the outer iteration complexity and conclude this subsection with total oracle complexity.
\subsubsection{Outer iteration complexity of \textbf{ManIAL}}
Without specifying a subroutine to obtain $x^{k+1}$, we establish the outer iteration complexity
result of Algorithm \ref{alg-manial} in the following lemma.
\begin{theorem}[Iteration complexity of \textbf{ManIAL} with \textbf{Option I}]\label{the:milam:outer}
    Suppose that Assumption \ref{assum} holds. For  $b>1$,  if $\sigma_k = b^k$ and $\epsilon_k = 1/\sigma_k$, given $\epsilon>0$, then Algorithm \ref{alg-manial} with \textbf{Option I} needs at most $K:= \log_{b}\left(\frac{\ell_h + \|z_{\max}\|+1}{ \epsilon}\right)$ iterations to produce an $\epsilon$-KKT solution pair $(x^{K+1},y^{K+1},\bar{z}^{K+1})$ of \eqref{eq:epsi-kkt}, where 
    \begin{equation}\label{def:zbar}
       \bar{z}^{k+1}: = z^k -  \sigma_k(\mathcal{A}x^{k+1} - y^{k+1}), \forall k\geq 0, ~~ z_{\max}: = \frac{\beta_0 \pi^2 }{6} \|\mathcal{A}x^0 - y^0\|.  
    \end{equation}
\end{theorem}

\begin{proof}
    It follows from the update rule of $x^{k+1}$ in Algorithm \ref{alg-manial} with \textbf{Option I} and \eqref{eq:grad-psik} that
    \begin{equation}\label{eq:x-epsilonk}
    \left\|\mathcal{P}_{T_x\mathcal{M}} \left( 
 \nabla f(x^{k+1}) + \sigma_k \mathcal{A}^*\left( 
\mathcal{A}x^{k+1} - z^k/\sigma_k - \mathrm{prox}_{h/\sigma_k}(
\mathcal{A}x^{k+1} - z^k/\sigma_k) \right) \right)  \right\| \leq \epsilon_k.
\end{equation}
Moreover, the update rule of $y^{k+1}$ implies that
\begin{equation}\label{eq:y-epsilonk}
   0 \in z^k -  \sigma_k(\mathcal{A}x^{k+1} - y^{k+1}) + \partial h(y^{k+1}). 
\end{equation}
Combining \eqref{eq:x-epsilonk}, \eqref{eq:y-epsilonk} and the definition of $\bar{z}^{K+1}$ in \eqref{def:zbar} yields
 \begin{equation}\label{eq:kkt-x-y}
 \left\{
        \begin{aligned}
\left\|\mathcal{P}_{T_{x^{k+1}}\mathcal{M}} \left( 
 \nabla f(x^{k+1}) + \mathcal{A}^*  \bar{z}^{k+1} \right)  \right\| \leq \epsilon_k, \\
\mathrm{dist}(-\bar{z}^{k+1}, \partial h(y^{k+1})) = 0.
    \end{aligned}
    \right.
    \end{equation}
Before bounding the feasibility condition $\|\mathcal{A}x^{k+1} - y^{k+1}\|$, we first give a uniform upper bound of the dual variable $z^{k+1}$. By \eqref{iter:beta}, \eqref{iter:z}, and $z^0 = 0$, we have that for any $k>1$,
\begin{equation}
\begin{aligned}
    \|z^{k+1}\| & \leq \sum_{l=1}^{k+1} \beta_l \|\mathcal{A}x^l - y^l\|
 \leq \sum_{l=1}^{\infty}\beta_l \|\mathcal{A}x^l - y^l\| \\
& \leq  \beta_0 \|\mathcal{A}x^0 - y^0\| \log^2 2 \sum_{l=1}^{\infty} \frac{1}{(l+1)^2\log(l+2)} \\
&\leq  \frac{\beta_0 \pi^2 }{6} \|\mathcal{A}x^0 - y^0\| = z_{\max},
\end{aligned}
\end{equation}
where the last inequality utilize that $\log(l+2)>1$ for $l\geq 1$ and $\sum_{l=1}^{\infty} \frac{1}{(l+1)^2} = \frac{\pi^2}{6}$.  
For the feasibility bound $\|\mathcal{A}x^{k+1} - y^{k+1}\|$, we have that for all $k\geq 1$,
\begin{equation}
    \begin{aligned}
        \|\mathcal{A}x^{k+1} - y^{k+1}\| & \leq \|\mathcal{A}x^{k+1} - z^{k}/\sigma_k - y^{k+1}\| + \frac{1}{\sigma_k}\|z^k\| \\
        & \leq \frac{\ell_h + \|z^k\|}{\sigma_k} \leq \frac{\ell_h + \|z_{\max}\|}{\sigma_k},
    \end{aligned}
\end{equation}
where the second inequality utilizes \eqref{eq:moreau-gradient-bound}. By the definition of $\sigma_k, \epsilon_k$ and $K$, we have that 
\begin{equation}\label{eq:kkt-x-y-z}
 \left\{
        \begin{aligned}
\left\|\mathcal{P}_{T_{x^{K+1}}\mathcal{M}} \left( 
 \nabla f(x^{K+1}) + \mathcal{A}^*  \bar{z}^{K+1} \right)  \right\| &\leq \epsilon, \\
\mathrm{dist}(-\bar{z}^{K+1}, \partial h(y^{K+1})) &= 0,\\
\| \mathcal{A}x^{K+1} - y^{K+1}\| &\leq \epsilon.
    \end{aligned}
    \right.
    \end{equation}
Therefore, $(x^{K+1},y^{K+1},\bar{z}^{K+1})$ is an $\epsilon$-KKT point of \eqref{eq:epsi-kkt}. 
\end{proof}

Analogously, we also have the following outer iteration complexity for \textbf{ManIAL} with \textbf{Option II}.
\begin{theorem}[Iteration complexity of \textbf{ManIAL} with \textbf{Option II}]\label{manial2-iter}
Suppose that Assumption \ref{assum} holds.   Let $\sigma_k = \frac{1}{2^{k/3}}$.  Given $\epsilon>0$, Algorithm \ref{alg-manial} with \textbf{Option II} needs at most $$K := \log_2\left( \frac{(\sqrt{2c_1}+ \ell_h + z_{\max})^3}{\epsilon^3} \right)$$ iterations to produce an $\epsilon$-KKT point pair $(x^{K+1},y^{K+1},\bar{z}^{K+1})$, where $c_1$ is defined in the proof. $\bar{z}^{K+1}$ and $z_{\max}$ are defined in \eqref{def:zbar}. 
\end{theorem}

\begin{proof}
    Similar to Theorem \ref{the:milam:outer}, for any $k\in \mathbb{N}$, we have that
    \begin{equation}\label{eq:temp-y-z}
        \begin{aligned}
\mathrm{dist}(-\bar{z}^{k+1}, \partial h(y^{k+1}))  = 0, ~~\|\mathcal{A}x^{k+1}  - y^{k+1}\| & \leq \frac{\ell_h + \|z_{\max}\|}{\sigma_k}.
    \end{aligned}
    \end{equation}
    Moreover, it follows from Lemma \ref{theo:rgd} that 
    \begin{equation}\label{eq:temp-x}
        \begin{aligned}
            \left\|\mathcal{P}_{T_{x^{k+1}}\mathcal{M}} \left( 
 \nabla f(x^{k+1}) + \mathcal{A}^*  \bar{z}^{k+1} \right)  \right\| & \leq \frac{\sqrt{L_k(\psi_k(x_k) - \psi_{k,\min})}}{\sqrt{2^k}}.
        \end{aligned} 
    \end{equation}
Next, we show that $\psi_k(x_k) - \psi_{k,\min}$ is bounded for any $k\geq 1$.  By the properties of the Moreau envelope, we have that 
    \begin{equation}
        \psi_k(x) \leq f(x) + h(\mathcal{A}x - z^k/\sigma_k) \leq \psi_k(x) +   \frac{\ell_h^2 + \|z_{\max}\|^2}{2\sigma_0}.
    \end{equation}
    Since $\Mcal$ is a compact submanifold, $z^k$ is bounded by $z_{\max}$ and $\sigma_k>\sigma_0$ has the low bound, there exist $\overline{\phi}$ and $\underline{\phi}$ such that  $\underline{\phi} \leq  f(x) + h(\mathcal{A}x - z^k/\sigma_k) \leq \overline{\phi}$ for any $x\in \Mcal$ and $k\geq 1$. Therefore, we have that
    \begin{equation}
        \psi_k(x^k) - \psi_{k,\min} \leq \overline{\phi} - \underline{\phi} + \frac{\ell_h^2 + \|z_{\max}\|^2}{2\sigma_0}.
    \end{equation}
 Let us denote $c_0 = \overline{\phi} - \underline{\phi} + \frac{\ell_h^2 + \|z_{\max}\|^2}{2\sigma_0}, c_1 = c_0\alpha^2 \|\mathcal{A}\|^2, c_2 = c_0(\alpha^2 \ell_{\nabla f}+2G\beta)$. By the fact that $\sigma_k = 2^{k/3}$, we have
 \begin{equation}\label{eq:temp-c1}
     \begin{aligned}
        & L_k(\psi_k(x_k) - \psi_{k,\min})  \\
        =& (\alpha^2 (\ell_{\nabla f} + \sigma_k \|\mathcal{A}\|^2 )  +2 G \beta)(\psi_k(x_k) - \psi_{k,\min}) \\
          =& c_1 2^{k/3} + c_2.
     \end{aligned}
 \end{equation}
 Combining with \eqref{eq:temp-y-z} and \eqref{eq:temp-x} yields  
     \begin{equation}
      \left\{  \begin{aligned}
\left\|\mathcal{P}_{T_{x^{k+1}}\mathcal{M}} \left( 
 \nabla f(x^{k+1}) + \mathcal{A}^*  \bar{z}^{k+1} \right)  \right\| & \leq 
\frac{\sqrt{c_1 2^{k/3} + c_2}}{ 2^{k/2}}, \\
\|\mathcal{A}x^{k+1}  - y^{k+1}\| &\leq  \frac{\ell_h + \|z_{\max}\|}{ 2^{k/3}},\\
\mathrm{dist}(-\bar{z}^{k+1}, \partial h(y^{k+1})) & = 0.
    \end{aligned}
    \right.
    \end{equation}
 Without loss of generality, we assume that $c_1 2^{k/3} \geq c_2$. This implies that $\frac{\sqrt{c_1 2^{k/3} + c_2}}{ 2^{k/2}} \leq \frac{\sqrt{2c_1}}{ 2^{k/3}}$.       Letting $K = \log_2\left( \frac{(\sqrt{2c_1}+ \ell_h + z_{\max})^3}{\epsilon^3} \right)$, one can shows that 
         \begin{equation}
         \left\{
        \begin{aligned}
\left\|\mathcal{P}_{T_{x^{K+1}}\mathcal{M}} \left( 
 \nabla f(x^{K+1}) + \mathcal{A}^*  \bar{z}^{K+1} \right)  \right\| & \leq \epsilon, \\
 \|\mathcal{A}x^{K+1}  - y^{K+1}\| &\leq  \epsilon,\\
\mathrm{dist}(-\bar{z}^{K+1}, \partial h(y^{K+1})) & = 0.
    \end{aligned}
    \right.
    \end{equation}
    This implies that $(x^{K+1},y^{K+1},\bar{z}^{K+1})$ is an $\epsilon$-KKT point pair. 
\end{proof}

\subsubsection{Overall oracle complexity of \textbf{ManIAL}}
Before giving the overall oracle complexity of \textbf{ManIAL},  we give the definition of a first-order oracle for \eqref{prob}.
\begin{definition}[\textbf{first-order oracle}]\label{def:detern}
    For problem \eqref{prob}, a first-order oracle can
be defined as follows: compute the Euclidean gradient $\nabla f(x)$, the proximal operator $\prox_h(x)$, and the retraction operator $\mathcal{R}$.  
\end{definition}
Now, let us combine the inner oracle complexity and the outer iteration complexity to derive the overall oracle complexity of \textbf{ManIAL}.
\begin{theorem}[Oracle complexity of \textbf{ManIAL}]\label{theo:oracle}
  Suppose that Assumption \ref{assum} holds. We use Algorithm \ref{alg4} as the subroutine to solve the subproblems in \textbf{ManIAL}. The following holds:
 \begin{itemize}
    \item[(a)] If we set $\sigma_k = b^k$ for some $b> 1$ and $\epsilon_k =1 /\sigma_k$, then \textbf{ManIAL} with \textbf{Option I} finds an $\epsilon$-KKT point, after at most $\tilde{ \mathcal{O}}(\epsilon^{-3})$ calls to the first-order oracle.
    \item[(b)] If we set $\sigma_k = \frac{1}{2^{k/3}}$, then \textbf{ManIAL} with \textbf{Option II} finds an $\epsilon$-KKT point, after at most $\mathcal{O}(\epsilon^{-3})$ calls to the first-order oracle.
 \end{itemize}
\end{theorem}

\begin{proof}
Let $K$ denote the number of (outer) iterations of Algorithm \ref{alg-manial} to reach the accuracy $\epsilon$. Let us show (a) first.
It follows from Theorem \ref{the:milam:outer} that $K = \log_{b}\left(\frac{\ell_h + \|z_{\max}\| + 1}{ \epsilon}\right)$. By Lemma \ref{theo:rgd}, in the $k$-th iterate, one needs at most $\frac{L_k(\psi_k(x^k) - \psi_{k,\min} )}{\epsilon_k^2}$ iterations of Algorithm \ref{alg4}.   Therefore, one can bound $T$:
    \be 
\begin{aligned}
    T = & \sum_{k = 1}^K \frac{L_k(\psi_k(x^k) - \psi_{k,\min} )}{\epsilon_k^2} \\
    \leq & \sum_{k = 1}^K \frac{c_1 \sigma_k + c_2}{\epsilon_k^2}  \leq  \sum_{k = 1}^K \left( \frac{c_1}{\epsilon_k^3} + \frac{c_2}{\epsilon_k^2}\right) \\
    \leq & K \left( \frac{c_1}{\epsilon^3} + \frac{c_2}{\epsilon^2}\right)\\ 
    \leq & \log_{b}\left(\frac{\ell_h + \|z_{\max}\| + 1}{ \epsilon}\right)\left( \frac{c_1}{\epsilon^3} + \frac{c_2}{\epsilon^2}\right),
\end{aligned}
    \ee
where the first inequality follows from \eqref{eq:temp-c1}. 
   Secondly, it follows from Theorem \ref{manial2-iter} that $K = \log_2\left( \frac{(\sqrt{2c_1}+ \ell_h + z_{\max})^3}{\epsilon^3} \right)$. Then we have
  \begin{equation}
    \begin{aligned}
    T = &   \sum_{k = 1}^K 2^k = \sum_{k = 1}^K (2^{k+1} - 2^k) \\
    = & 2^{K+1} - 2 =2 \left( \frac{\left(\sqrt{2c_1}+ \ell_h + z_{\max}\right)^3}{\epsilon^3} -1   \right). 
\end{aligned}
  \end{equation}
    The proof is completed.
\end{proof}

     It should be emphasized that \textbf{ManIAL} with \textbf{Option I}, as well as the existing works, such as \cite{lin2019inexact} and \cite{lin2019inexact}, achieves an oracle complexity of $\mathcal{O}(\epsilon^{-3})$ up to a logarithmic factor. However, \textbf{ManIAL} with \textbf{Option II} achieves an oracle complexity of $\mathcal{O}(\epsilon^{-3})$ independently of the logarithmic factor. In that sense, we demonstrate that \textbf{ManIAL} with \textbf{Option II} achieves a better oracle complexity result.

\section{Stochastic augmented Lagrangian method}
In this section, we focus on the case that $f$ in problem \eqref{prob} has the expectation form, i.e., $f(x) = \mathbb{E}_{\xi\in \mathcal{D}}[f(x,\xi)]$. In particular, we consider the following optimization problem on manifolds:
\begin{equation}\label{prob:stochastic}
    \min_{x\in\mcM} F(x):= \mbE_{\xi\in\mcD}[f(x,\xi)] + h(\mcA x).
\end{equation}
Here, we assume that the function $h$ is convex and $\ell_h$-Lipschitz continuous, $\nabla f(x,\xi)$ is $\ell_{\nabla f}$-Lipschitz continuous.  The following assumptions are made for the stochastic gradient $\nabla f(x,\xi)$:
\begin{assumption}\label{assum:stochatis}
For any $x$, the algorithm generates a sample $\xi\sim \mathcal{D}$ and returns a stochastic gradient $\nabla f(x,\xi)$, there exists a parameter $\delta>0$ such that
\begin{align}
  &  \mathbb{E}_{\xi}\left[ \nabla f(x,\xi) \right] = \nabla f(x), \label{assm:mean} \\ 
 &    \mathbb{E}_{\xi}\left[ \|\nabla f(x,\xi) - \nabla f(x) \|^2\right] \leq \delta^2. \label{assm:vari}
\end{align}
Moreover, $\nabla f(x,\xi)$ is $\ell_{\nabla f}$-Lipschitz continuous.
\end{assumption}

Similarly, we give the definition of $\epsilon$-stationary point for problem \eqref{prob:stochastic}:
\begin{definition}
We say $x\in\mathcal{M}$ is an $\epsilon$-stationary point of \eqref{prob:stochastic} if there exists $y,z\in\mathbb{R}^m$ such that 
\begin{equation}\label{eq:epsi-kkt-2}
    \left\{
\begin{aligned}
    \mathbb{E}_{\xi}\left[\left\|\mathcal{P}_{T_x\mathcal{M}}\left(\nabla f(x,\xi) - \mathcal{A}^*z \right) \right\| \right]\leq \epsilon, \\
   \mathrm{dist}( -z, \partial h(y)) \leq 
\epsilon, \\
   \| \mathcal{A}x - y\| \leq \epsilon.
\end{aligned}
    \right.
\end{equation}
In other words, $(x,y,z)$ is an $\epsilon$-KKT point pair of \eqref{prob:stochastic}. 
\end{definition}

\subsection{Algorithm}
We define the augmented Lagrangian function of \eqref{prob:stochastic}:
\begin{equation}
\begin{aligned}
     \mathcal{L}_{\sigma}(x,z): &= \min_{y\in \mathbb{R}^m} \left\{ \mathbb{E}_{\xi} [f(x,\xi) ] + h(y) - \langle z, Ax-y \rangle + \frac{\sigma}{2}\|Ax - y\|^2\right\} \\
     & = \mathbb{E}_{\xi} [f(x,\xi) ] + M_{h}^{1/\sigma}(Ax - z/\sigma) - \frac{1}{2\sigma}\|z\|^2.
\end{aligned}
\end{equation}
To efficiently find an $\epsilon$-stationary point of \eqref{prob:stochastic}, we design a Riemannian stochastic gradient-type method based on the framework of the stochastic manifold inexact augmented Lagrangian method (\textbf{StoManIAL}), which is given in
Algorithm \ref{alg:stomanial}. 
To simplify the notation, in the $k$-iterate, let us denote
\begin{equation}
    \psi_k(x,\xi): = f(x,\xi) + M_{h}^{1/\sigma_k}(Ax - z^k/\sigma_k) - \frac{1}{2\sigma_k}\|z^k\|^2, 
\end{equation}
and $\psi_k(x): = \mathbb{E}_{
\xi
}[\psi_k(x,\xi)]$.  Under Assumption \ref{assum:stochatis}, it can be easily shown that
\begin{equation}\label{def:psik-delta}
    \mathbb{E}_{\xi}\|\psi_k(x,\xi) - \psi_k(x)\|^2 \leq \delta^2. 
\end{equation}
Moreover, as shown in \eqref{eq:grad-psik}, one can derive the Euclidean gradient of $\psi_k(x,\xi)$ as follows
\begin{equation}\label{eq:grad-psik-1}
\begin{aligned}
  \nabla \psi_k(x,\xi) & = \nabla f(x,\xi) + \sigma_k \mathcal{A}^*\left(\mathcal{A}x-\frac{1}{\sigma_k} z^k - \mbox{prox}_{ h/\sigma_k}(\mathcal{A}(x)-\frac{1}{\sigma_k} z^k)\right).
  \end{aligned}
\end{equation}

Note that, unlike \textbf{ManIAL}, the condition \eqref{eq:solve-x-1} of \textbf{Option I} in \textbf{StoManIAL} cannot be checked in the numerical experiment due to the expectation operation, although its validity can be assured by the convergence rule result of the subroutine, which will be presented in the next subsection. In contrast, the condition in \textbf{Option II} is checkable. 

\begin{algorithm}[H]
\footnotesize
\caption{ (\textbf{StoManIAL})}\label{alg:stomanial}
\begin{algorithmic}[1]
\REQUIRE Initial point $x^0\in \mathcal{M}, y^0, z^0$;  $\epsilon$, $\{\sigma_k\}_k$. Set $k=0$. 
\WHILE {not converge}
\STATE Solve the following manifold optimization problem
   \be \label{prob:sto-x-sub}
 \min_{x\in \Mcal} \mathcal{L}_{\sigma_k}(x,z^k)
   \ee
   and return $x^{k+1}$ with two options:
   \begin{itemize}
       \item \textbf{Option I}: $x^{k+1}$ satisfies 
        \begin{equation}\label{eq:solve-x-1}
  \mathbb{E}_{\xi}\left[ \left \|\grad \psi_k(x^{k+1},\xi)\right\| \right] \leq \epsilon_k.
\end{equation}
\item \textbf{Option II}: return $x^{k+1}$ with the fixed iteration number $2^k$. 
   \end{itemize}
\STATE Update the auxiliary variable: $$ y^{k+1} = \mathrm{prox}_{h/\sigma_k}(\mathcal{A}x^{k+1} - z^k/\sigma_k).$$
\STATE Update the step sizes:
   $$
       \beta_{k+1} = \beta_1\min\left(\frac{\|\mathcal{A}x^1 - y^1\|\log^22}{\|\mathcal{A}x^{k+1}-y^{k+1}\|(k+1)\log^2(k+2)},1\right).
  $$
\STATE Update the dual variable: \begin{equation}
       z^{k+1} = z^k -  \beta_{k+1}(\mathcal{A}x^{k+1} - y^{k+1}).
   \end{equation}
\STATE Set $k=k+1$.
\ENDWHILE
\end{algorithmic}
\end{algorithm}

\subsection{Subproblem solver}

In this subsection, we present a Riemannian stochastic optimization method to solve the $x$-subproblem in Algorithm \ref{alg:stomanial}. As will be shown in Lemma \ref{lem:deter:L},  the function $\mathcal{L}_{\sigma_k}(x,z^k)$ is retr-smooth with respect to $x$, and its retr-smoothness parameter depends on $\sigma_k$, which, in turn, is ultimately determined by a prescribed tolerance $\epsilon_k$. To achieve a favorable overall complexity outcome with a lower order, it is essential to employ a subroutine whose complexity result exhibits a low-order dependence on the smoothness parameter $\sigma_k$.

A recent work, the momentum-based variance-reduced stochastic gradient method known as Storm \cite{cutkosky2019momentum}, has demonstrated efficacy in solving nonconvex smooth stochastic optimization problems in Euclidean space, yielding a low-order overall complexity result. Despite the existence of a Riemannian version \cite{han2020riemannian} of Storm (RSRM) that matches the same overall complexity result, its dependence on the smoothness parameters is excessively significant for this outcome. Therefore, we introduce a modified Riemannian version of Storm. Subsequently, in the next subsection, we elaborate on its application in determining $x^{k+1}$ within Algorithm \ref{alg:stomanial}.

Consider the following general stochastic optimization problem on a Riemannian manifold
\begin{equation}
    \min_{x\in \mathcal{M}} \mathbb{E}[\psi(x,\xi)].
\end{equation}
Suppose the following conditions hold: for some finite constants $L, \delta$
\begin{align}
  &  \mathbb{E}_{\xi}\left[ \nabla \psi(x,\xi) \right] = \nabla \psi(x), ~~     \mathbb{E}_{\xi}\left[ \|\nabla \psi(x,\xi) - \nabla \psi(x) \|^2\right] \leq \delta^2,\label{assum-rsrm-1} \\
 & \|\grad \psi(x,\xi) - \mathcal{T}_{y}^x \grad \psi(y,\xi) \| \leq L \| d\|,~~ \forall x, y\in \mathcal{M}, d = \mathcal{R}_x^{-1}(y),\label{lipschi-rie} \\
&  \psi(y,\xi)  \leq  \psi(x,\xi) + \langle d, \grad \psi(x,\xi) \rangle + \frac{L}{2}\|d\|^2, ~~ \forall x, y\in \mathcal{M}, d = \mathcal{R}_x^{-1}(y). \label{assum-rsrm-3}
\end{align}

With the above conditions, we give the modified Riemannian version of Storm in Algorithm \ref{alg:rsrm} and the
complexity result in Lemma \ref{lem:rsrm}. The proof is similar to Theorem 1 in \cite{cutkosky2019momentum}. We provide it in the Appendix.

\begin{algorithm}[H]
\caption{ Riemannian stochastic recursive momentum gradient descent method: $\text{RStorm}(\psi,L,T,x_1)$}\label{alg:rsrm}
\begin{algorithmic}[1]
\REQUIRE Initial point $x_{1}$, iteration limit $T$, parameters $\kappa,w,c$, sample set $\{\xi_t\}_{t\geq 1}$. \\
\STATE Compute $d_1 = -\grad  \psi(x_1,\xi_1)$
\FOR{$t=1, \ldots, T$}
\STATE Update $\eta_t = \frac{\kappa}{(w+\sum_{i=1}^t G_t^2)^{1/3}}$
 \STATE  Update $x_{t+1}$ via $
x_{t+1} = \mathcal{R}_{x_t}( -\eta_t d_t).$
\STATE Update $a_{t+1} = c\eta_t^2$ and $G_{t+1} = \|\grad \psi(x_{t+1},\xi_{t+1})\|$.
\STATE Compute $d_{t+1} = \grad  \psi(x_{t+1};\xi_{t+1}) + (1-a_{t+1}) \mathcal{T}_{x_t}^{x_{t+1}}(d_t - \grad  \psi(x_t;\xi_t)) $. 
\ENDFOR
\STATE \textbf{Output:} Return $\hat{x}$ uniformly at random from $x_1,\cdots,x_T$. 
\end{algorithmic}
\end{algorithm}

\begin{lemma}\label{lem:rsrm}
    Suppose that the conditions \eqref{assum-rsrm-1}-\eqref{assum-rsrm-3} hold. Let us denote $G = \max_{x\in \Mcal} \nabla \psi(x,\xi)$. For any $b>0$, we write $\kappa=\frac{b G^{\frac{2}{3}}}{L}$. Set $c=10 L^2+$ $G^2 /\left(7 L \kappa^3\right)$ and $w=\max \left((4 L \kappa)^3, 2 G^2,\left(\frac{c \kappa}{4 L}\right)^3\right)$.
Then, the output $\hat{x}$ in Algorithm \ref{alg:rsrm} satisfies
\begin{equation}\label{eq:grad-rsrm-bound}
\mathbb{E}[\|\grad  \psi(\hat{x})\|]=\mathbb{E}\left[\frac{1}{T} \sum_{t=1}^T\left\|\grad \psi\left(x_t\right)\right\|\right] \leq \frac{w^{1 / 6} \sqrt{2 M}+2 M^{3 / 4}}{\sqrt{T}}+\frac{2 \delta^{1 / 3}}{T^{1 / 3}},
\end{equation}
where $M=\frac{6}{\kappa}\left(\psi\left(x_1\right)-\psi_{\min}\right)+\frac{w^{1 / 3} \delta^2}{2 L^2 \kappa^2}+\frac{\kappa^2 c^2}{ L^2} \ln (T+2)$.
\end{lemma} 

To apply Algorithm \ref{alg:rsrm} to find $x^{k+1}$, we need to validate conditions \eqref{assum-rsrm-1}-\eqref{assum-rsrm-3} for the associated subproblem. 
\begin{lemma}\label{lem:deter:L}
   Suppose that Assumption \ref{assum} holds.  Then, for $\psi_k$, there exists finite $G$ such that $\|\nabla \psi_k(x,\xi)\|\leq G$ for all $x\in\mathcal{M}$. Moreover, the function $\psi_k(x,\xi)$ is retr-smooth with constant $\hat{L}_k$ and the Riemannian gradient $\grad \psi_k(x,\xi)$ is Lipschitz continuous with constant $\bar{L}_k$, where
    \begin{equation} \label{eq:sto-retr-smooth}
    \hat{L}_k: = \alpha^2 \ell_{\nabla f} + \alpha^2 \|\mathcal{A}\|\sigma_k + 2G\beta,\; \bar{L}_k: = (\alpha L_p + \zeta) G + \alpha  \ell_{\nabla f} + \|\mathcal{A}\|\sigma_k,  
    \end{equation}
    and $L_p$ and $\zeta$ are defined in Lemma \ref{lem:eucli-rieman-lipsch}. 
    Moreover, under Assumption \ref{assum:stochatis}, conditions \eqref{assum-rsrm-1}-\eqref{assum-rsrm-3} with constants $L_k:= \max\{\hat{L}_k,\bar{L}_k\}$ and $\delta$ hold for subproblem \eqref{prob:sto-x-sub}.  
\end{lemma}

\begin{proof}
   We first show \eqref{eq:sto-retr-smooth}. By Proposition \ref{propos-1}, one shows that $\nabla \psi_k$  is Lipschitz continuous with constant $ \ell_{\nabla f} + \sigma_k \|\mathcal{A}\|^2 $. The retraction smoothness of $\psi_k(x,\xi)$ can directly follow from Lemma \ref{Euclidean vs manifold}. The Lipschitz continuity of $\grad \psi_k(x,\xi)$ follows from  Lemma \ref{lem:eucli-rieman-lipsch}. These induce conditions \eqref{lipschi-rie} and \eqref{assum-rsrm-3}. Under Assumption \ref{assum:stochatis}, condition \eqref{assum-rsrm-1} is directly follows from the definition of $\psi_k(x)$ and \eqref{def:psik-delta}. We complete the proof.
\end{proof}

As shown in Lemma \ref{lem:deter:L}, the Lipschitz constant $L_k$ increases as the penalty parameter $\sigma_k$ increases. Therefore, we analyze the dependence of $L$ for the right-hand side in \eqref{eq:grad-rsrm-bound} when $L$ tends to infinity. Noting that $\kappa = \mathcal{O}(1/L), c = \mathcal{O}(L^2), w = \mathcal{O}(1)$, and $M = \mathcal{O}(L)$, we can simplify the expression of \eqref{eq:grad-rsrm-bound} via $\mathcal{O}$ notation:
\begin{equation}
    \mathbb{E}[\|\grad  \psi(\hat{x})\|] \leq \tilde{\mathcal{O}}\left(\frac{L^{1/2} + L^{3/4}}{T^{1/2}}  + \frac{1}{T^{1/3}}   \right),
\end{equation}
where $\tilde{\mathcal{O}}$ is due to the dependence of $\ln(T)$.

With Lemma \ref{lem:deter:L}, we are able to apply Algorithm \ref{alg:rsrm} to find $x^{k+1}$. The lemma below gives the oracle complexity for the $k$-th outer iteration of Algorithm \ref{alg:stomanial}. This is a direct application of Lemmas \ref{lem:rsrm} and \ref{lem:deter:L}, and we omit the proof.

\begin{lemma}\label{theo:rsrm-1}
   Suppose that Assumptions \ref{assum} and \ref{assum:stochatis} hold. Let $(x^k,y^k,z^k)$ be the $k$-iterate generated in Algorithm \ref{alg:stomanial}. We can find $x^{k+1}$ with \textbf{Option I} by the following call
    \begin{equation}
      x^{k+1} = \text{RStorm}(\psi_k, L_k, T_k,x^k),
    \end{equation}
    where 
    $
    L_k =\max\{\hat{L}_k,\bar{L}_k\},~~T_k:= \tilde{\mathcal{O}}\left(\frac{1}{\epsilon_k^{3}} + \frac{\sigma_k^{1.5} + \sigma_k}{\epsilon_k^2} \right). 
    $
    Moreover, we can also find $x^{k+1}$ with \textbf{Option II} by the following call
     \begin{equation}
      x^{k+1} = \text{RStorm}(\psi_k, L_k, 2^k, x^k),
    \end{equation}
This implies that
\begin{equation}
 \mathbb{E}_{\xi} \left[\| \grad \psi_k (x^{k+1},\xi) \| \right]\leq  \tilde{\mathcal{O}}\left( \frac{\sigma_k^{3/4} + \sigma_k^{1/2}}{2^{k/2}} + \frac{1}{2^{k/3}} \right).
\end{equation}
\end{lemma}

\subsection{Convergence analysis of \textbf{StoManIAL}}
We have shown the oracle complexity for solving the subproblems of \textbf{StoManIAL}. Now, let us first investigate the outer iteration complexity and conclude this subsection with the overall oracle complexity.
\subsubsection{Outer iteration complexity of \textbf{StoManIAL}} Without specifying a subroutine to obtain $x^{k+1}$, we establish the outer iteration complexity
result of Algorithm \ref{alg:stomanial} in the following lemma. Since the proof is similar to Theorem \ref{the:milam:outer}, we omit it. 
\begin{theorem}[Iteration complexity of \textbf{StoManIAL} with \textbf{Option I}]\label{lem:milam:outer-1}
    Suppose that Assumptions \ref{assum} and \ref{assum:stochatis} hold. For  $b>1$,  if $\sigma_k = b^k$ and $\epsilon_k = 1/\sigma_k$, given $\epsilon>0$, then Algorithm \ref{alg:stomanial} needs at most $K:= \log_{b}\left(\frac{\ell_h + \|z_{\max}\|+1}{ \epsilon}\right)$ iterations to produce an $\epsilon$-KKT solution pair $(x^{K+1},y^{K+1},\bar{z}^{K+1})$ of \eqref{eq:epsi-kkt}, where 
    \begin{equation}\label{def:zbar-1}
       \bar{z}^{k+1}: = z^k -  \sigma_k(\mathcal{A}x^{k+1} - y^{k+1}), \forall k\geq 0, ~~ z_{\max}: = \frac{\beta_0 \pi^2 }{6} \|\mathcal{A}x^0 - y^0\|.  
    \end{equation}
\end{theorem}

Analogous to Theorem \ref{lem:milam:outer-1}, we also have the following theorem on the iteration complexity of \textbf{StoManIAL} with a fixed number of inner iterations. 
\begin{theorem}[Iteration complexity of \textbf{StoManIAL} with \textbf{Option II}]\label{stomanial2-iter}
Suppose that Assumption \ref{assum} and \ref{assum:stochatis} hold.   Let $\sigma_k = \frac{1}{2^{2k/7}}$.  Given $\epsilon>0$, \textbf{StoManIAL} with \textbf{Option II} needs at most $\mathcal{O}(\log_2(\epsilon^{-7/2}))$ iterations to produce an $\epsilon$-KKT point pair $(x^{K+1},y^{K+1},\bar{z}^{K+1})$, where $\bar{z}^{K+1}$ is defined in \eqref{def:zbar-1}.
\end{theorem}

\begin{proof}
    Similar to Theorem \ref{the:milam:outer}, for any $k\in \mathbb{N}$, we have that
    \begin{equation}
        \begin{aligned}
\mathrm{dist}(-\bar{z}^{k+1}, \partial h(y^{k+1})) & = 0, \\
\|\mathcal{A}x^{k+1}  - y^{k+1}\| & \leq \frac{\ell_h + \|z_{\max}\|}{\sigma_k}.
    \end{aligned}
    \end{equation}
    Moreover, it follows from Lemma \ref{theo:rsrm-1} that 
    \begin{equation}
        \begin{aligned}
          \mathbb{E}_{\xi} \left[  \left\|\mathcal{P}_{T_{x^{k+1}}\mathcal{M}} \left( 
 \nabla f(x^{k+1},\xi) + \mathcal{A}^*  \bar{z}^{k+1} \right)  \right\|\right] & \leq \mathcal{O}\left(\frac{\sigma_k^{3/4} + \sigma_k^{1/2}}{2^{k/2}} + \frac{1}{2^{k/3}} \right).
        \end{aligned} 
    \end{equation}
   Since $\sigma_k = 2^{\frac{2k}{7}}$, we have that
     \begin{equation}
        \begin{aligned}
&  \mathbb{E}_{\xi} \left[  \left\|\mathcal{P}_{T_{x^{k+1}}\mathcal{M}} \left( 
 \nabla f(x^{k+1},\xi) + \mathcal{A}^*  \bar{z}^{k+1} \right)  \right\|\right]  \leq \mathcal{O}(2^{-2k/7} + 2^{-5k/14} + 2^{-k/3}),
    \end{aligned}
    \end{equation}
    and 
    $$
    \|\mathcal{A}x^{k+1}  - y^{k+1}\|  \leq  \mathcal{O}(2^{-2k/7}).
    $$
    Letting $K = \mathcal{O}(\log_2(\epsilon^{-\frac{7}{2}}))$, one can shows that 
         \begin{equation}
         \left\{
        \begin{aligned}
\mathbb{E}_{\xi} \left[\left\|\mathcal{P}_{T_{x^{K+1}}\mathcal{M}} \left( 
 \nabla f(x^{K+1},\xi) + \mathcal{A}^*  \bar{z}^{K+1} \right)  \right\| \right] & \leq \mathcal{O}(\epsilon), \\
\mathrm{dist}(-\bar{z}^{K+1}, \partial h(y^{K+1})) & = 0, \\
\|\mathcal{A}x^{K+1}  - y^{K+1}\| &\leq  \mathcal{O}(\epsilon),
    \end{aligned}
    \right.
    \end{equation}
    which implies that $(x^{K+1},y^{K+1},\bar{z}^{K+1})$ is an $\epsilon$-KKT point pair. 
\end{proof}

\subsubsection{Overall oracle complexity of \textbf{StoManIAL}}To measure the oracle complexity of our algorithm,  we give the definition of a stochastic first-order oracle for \eqref{prob:stochastic}.
\begin{definition}[\textbf{stochastic first-order oracle}]\label{def:stochastic}
    For the problem \eqref{prob:stochastic}, a stochastic first-order oracle can
be defined as follows: compute the Euclidean gradient $\nabla f(x,\xi)$ given a sample $\xi\in \mathcal{D}$, the proximal operator $\prox_h(x)$ and the retraction operator $\mathcal{R}$.  
\end{definition}
We are now able to establish the overall oracle complexity in the following theorem.
\begin{theorem}[Oracle complexity of \textbf{StoManIAL}]\label{theo:sto:oracle}
 Suppose that Assumptions \ref{assum} and \ref{assum:stochatis} hold. We choose Algorithm \ref{alg:rsrm} as the subroutine of \textbf{StoManIAL}. The following holds:
 \begin{itemize}
    \item[(a)] If we set $\sigma_k = b^k$ for some $b> 1$ and $\epsilon_k =1 /\sigma_k$, then \textbf{StoManIAL} with \textbf{Option I} finds an $\epsilon$-KKT point, after at most $\tilde{ \mathcal{O}}(\epsilon^{-3.5})$ calls to the stochastic first-order oracle.
    \item[(b)] If we set $\sigma_k = \frac{1}{2^{2k/7}}$, then \textbf{StoManIAL} with \textbf{Option II} finds an $\epsilon$-KKT point, after at most $\tilde{\mathcal{O}}(\epsilon^{-3.5})$ calls to the stochastic first-order oracle.
 \end{itemize}
\end{theorem}
\begin{proof}
 Let $K$ denote the number of (outer) iterations of Algorithm \ref{alg:stomanial} to reach the accuracy $\epsilon$. Let us show (a) first.
    It follows from Theorem \ref{lem:milam:outer-1} that $K = \log_{b}\left(\frac{\ell_h + \|z_{\max}\|}{\sigma_0 \epsilon}\right) + 1$. Combining this with Lemma \ref{theo:rsrm-1}, one can bounds the number $T$ of calls to the stochastic first-order oracle:
    \be 
\begin{aligned}
    T = & \sum_{k = 1}^K \tilde{\mathcal{O}}\left(\frac{1}{\epsilon_\kappa^3} + \frac{\sigma_k^{1.5} + \sigma_k}{\epsilon_k^{2}}\right) = \sum_{k = 1}^K \mathcal{O}\left(\frac{1}{\epsilon_k^3} + \frac{1}{\epsilon_k^{3.5}}\right) \\
     \leq & K \tilde{\mathcal{O}}\left(\frac{1}{\epsilon^3} + \frac{1}{\epsilon^{3.5}}\right) =\tilde{\mathcal{O}}(\log_b( \epsilon^{-1}) \epsilon^{-3.5}).
\end{aligned}
    \ee

   In terms of (b), it follows from Theorem \ref{stomanial2-iter} and Lemma \ref{theo:rsrm-1} that
  \begin{equation}
    \begin{aligned}
    T = & \sum_{k = 1}^K l_k = \sum_{k = 1}^K 2^k = \sum_{k = 1}^K (2^{k+1} - 2^k) \\
    = & 2^{k+1} - 2 =2 \tilde{\mathcal{O}}(  2^{\log_2(\epsilon^{-3.5})} -1   )= \tilde{\mathcal{O}}(\epsilon^{-3.5}). 
\end{aligned}
  \end{equation}
    The proof is completed.
\end{proof}




\section{Numerical results}

In this section, we demonstrate the performance of the proposed methods \textbf{ManIAL} and \textbf{StoManIAL} on the sparse principal component analysis. We use \textbf{ManIAL-I} and \textbf{ManIAL-II} to denote Algorithms \textbf{ManIAL} with option I and option II, respectively. Due to that \textbf{StoManIAL} with option I is not checkable, we only give the comparison of $\textbf{StoManIAL}$ with option II. We compare those algorithms with the subgradient method (we call it Rsub) in \cite{li2021weakly} that achieves the state-of-the-art oracle complexity for solving \eqref{prob:stochastic}. All the tests were performed in MATLAB 2022a on a ThinkPad X1 with 4 cores and 32GB memory. 
The following relative KKT residual of problem \eqref{prob} is set to a stopping criterion for our \textbf{ManIAL}:
\begin{equation}\label{kkt stop1}
     \text{error}: = \max\left\{\eta_p,\eta_d,\eta_C\right\} \leq \mbox{tol},
\end{equation}
where ``tol'' is a given accuracy tolerance and
\begin{equation}\label{equ:stop kkt}
  \begin{aligned}
     \eta_p &:= \frac{ \left \|\mathcal{A}x^k - y^k\right\|}{1+ \left\|\mathcal{A}x^k\|+\|y^k\right\|},\\
     \eta_d &: =\frac{\left\|\mathcal{P}_{T_{x^k}\mathcal{M}}(\nabla f(x^k) - \mathcal{A}^*z^k)\right\|}{1+\left\|\nabla f(x^k)\right\|}, \\
     \eta_C &: = \frac{ \| z^k - \mbox{prox}_{h^*}(  z^k -\mathcal{A}x^k)\|}{1+\|z^k\|}.
  \end{aligned}
\end{equation}

 In the following experiments, we will first run our algorithms \textbf{ManIAL-I} and \textbf{ManIAL-II}, and terminate them when either condition \eqref{kkt stop1} is satisfied or the maximum iteration steps of 10,000 are reached.  The obtained function value of \textbf{ManIAL-I} is denoted as $F_M$.  For the other algorithms, we terminate them when either the objective function value satisfies $F(x^k)\leq F_M + 10^{-10}$ or the maximum iteration steps of 10,000 are reached.
\subsection{Sparse principal component analysis}
Given a data set $\{b_1,\ldots, b_m\}$ where $b_i\in\mathbb{R}^{n\times 1}$, the sparse principal component analysis (SPCA) problem is
\begin{equation}\label{spca}
\begin{aligned}
  \min_{X\in\mathbb{R}^{n\times r}}   \sum_{i=1}^{m}\|b_i - XX^Tb_i\|_2^2 + \mu \|X\|_1, ~ ~ \mbox{s.t. }~~  X^TX = I_r,
  \end{aligned}
\end{equation}
where $\mu > 0$ is a regularization parameter. Let $B = [b_1,\cdots, b_m ]^{T}\in\mathbb{R}^{m\times n}$, problem \eqref{spca} can be rewritten as: 
\begin{equation}\label{spcaM}
  \begin{aligned}
     \min_{X\in\mathbb{R}^{n\times r}}   -\mathrm{tr}(X^TB^TBX) + \mu \|X\|_1,     ~~  \mbox{s.t. }~~  X^TX = I_r.
  \end{aligned}
\end{equation}
Here, the constraint consists of the Stiefel manifold $\texttt{St}(n,r):=\{X\in\mathbb{R}^{n\times r}~:~X^\top X = I_r\}$. The tangent space of $\texttt{St}(n,r)$ is defined by $T_{X}\texttt{St}(n,r) = \{\eta\in \mathbb{R}^{n\times r}~:~X^\top \eta + \eta^\top X = 0\}$. Given any $U\in\mathbb{R}^{n\times r}$, the projection of $U$ onto $T_{X}\texttt{St}(n,r)$ is $\mathcal{P}_{T_{X}\texttt{St}(n,r)}(U) = U - X \frac{U^\top X + X^\top U}{2}$ \cite{AbsMahSep2008}. In our experiment, the data matrix $B\in\mathbb{R}^{m\times n}$ is produced by MATLAB function $\texttt{randn}(m, n)$, in which all entries of $B$ follow the standard Gaussian distribution. We shift the columns of $B$ such that they have zero mean, and finally the column vectors are normalized. We use the polar decomposition as the retraction mapping. 

\begin{figure}[htpb]
\centering
\setlength{\abovecaptionskip}{0.cm}
\subfigure[$r = 10, \mu = 0.4$]{
\includegraphics[width=0.3\textwidth]{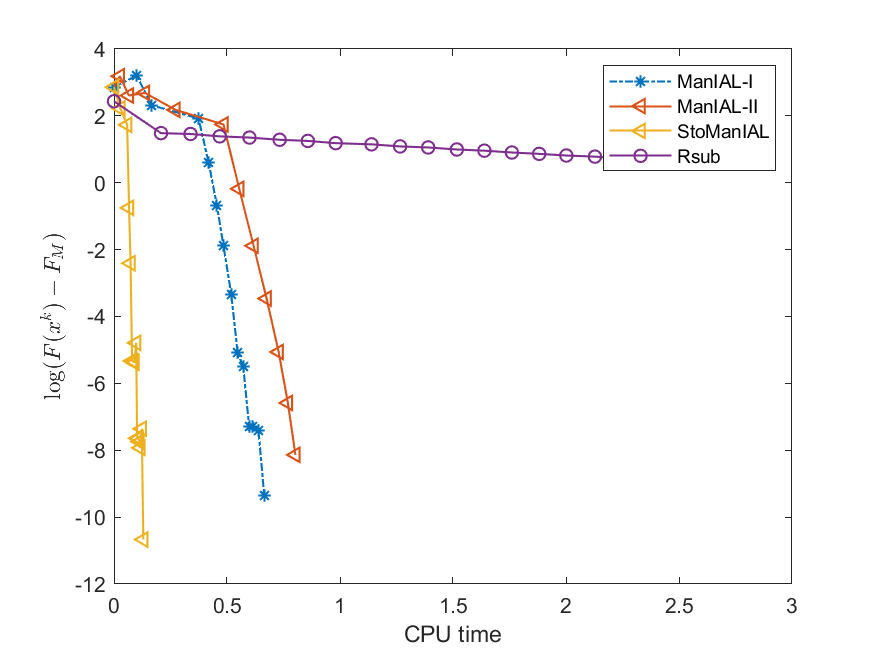}}
\subfigure[$r = 10, \mu = 0.6$]{
\includegraphics[width=0.3\textwidth]{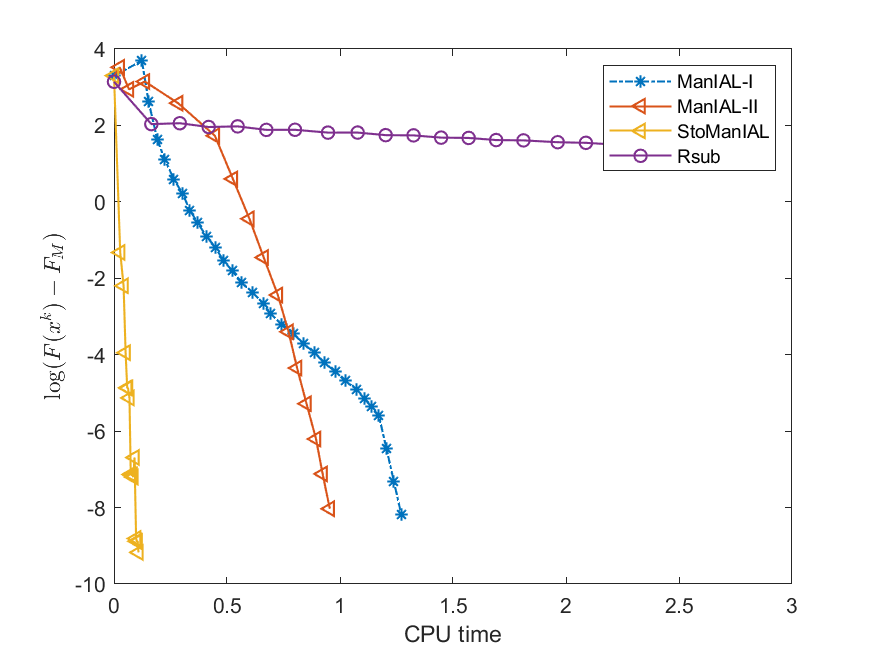}}
\subfigure[$r = 10, \mu = 0.8$]{
\includegraphics[width=0.3\textwidth]{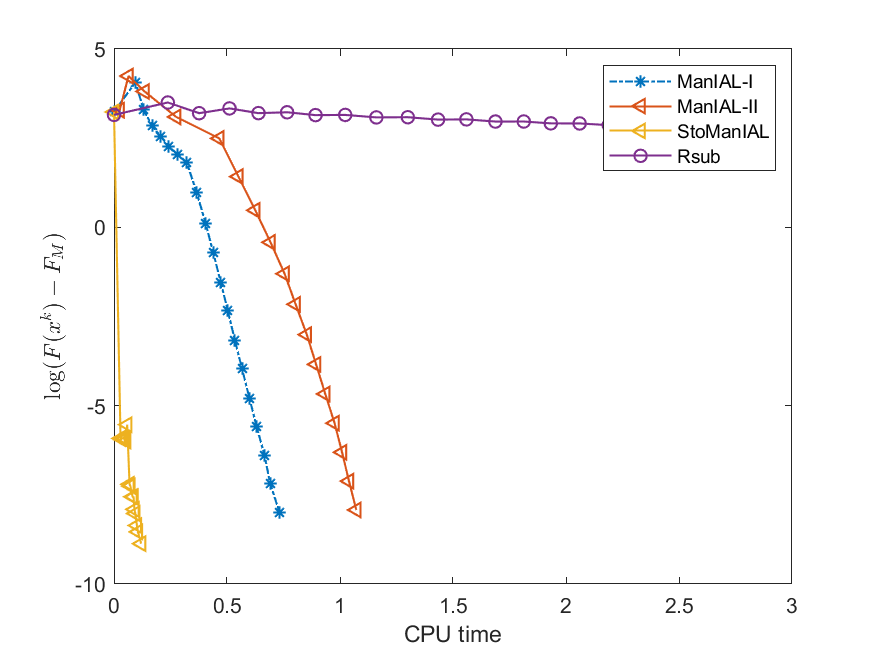}} \\
\subfigure[$r = 20, \mu = 0.4$]{
\includegraphics[width=0.3\textwidth]{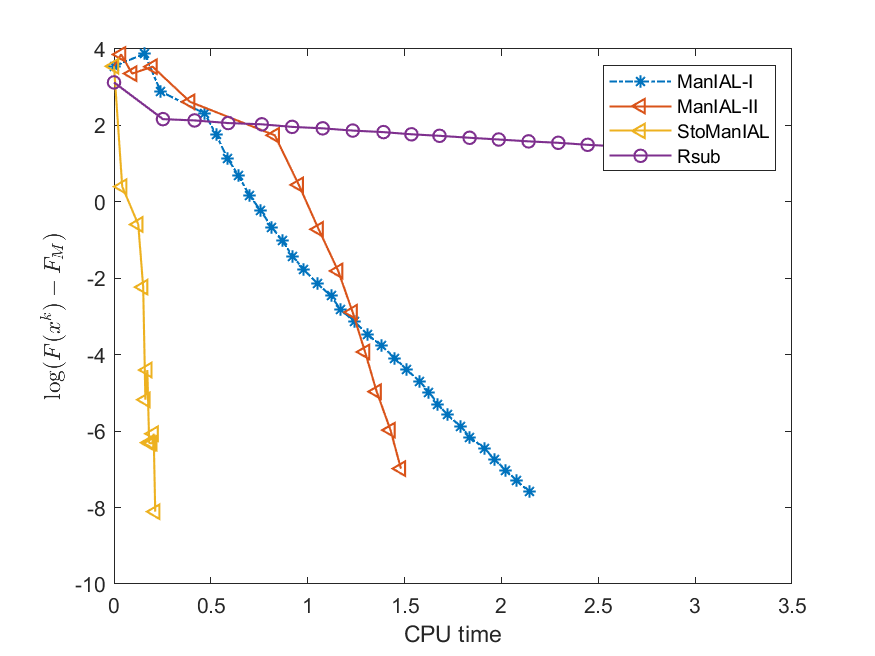}}
\subfigure[$r = 20, \mu = 0.6$]{
\includegraphics[width=0.3\textwidth]{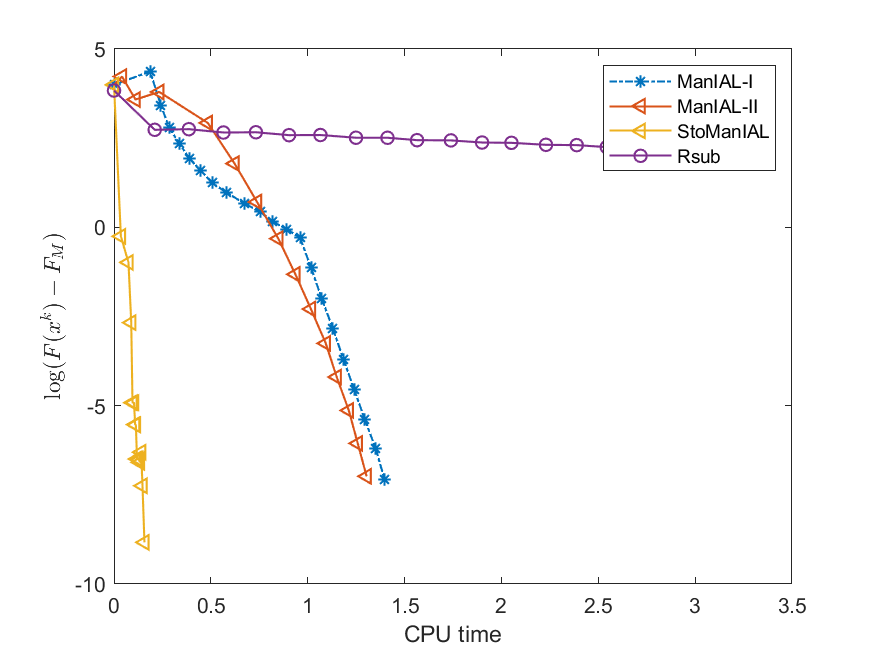}}
\subfigure[$r = 20, \mu = 0.8$]{
\includegraphics[width=0.3\textwidth]{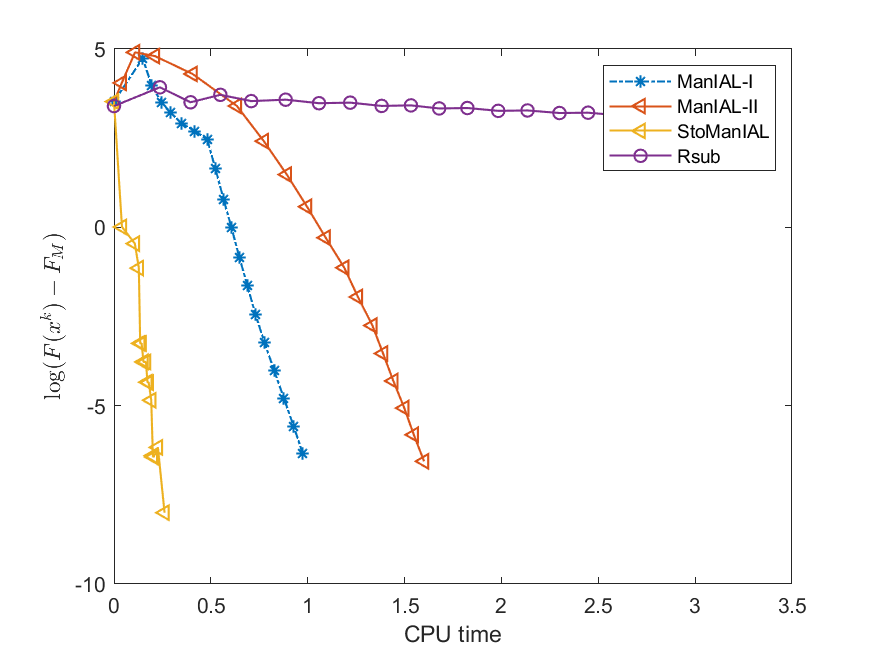}}
\caption{The performance of our algorithms on SPCA with random data $(n = 1000)$}
\label{fig:2}
\end{figure}
We first compare the performance of the proposed algorithms on the SPCA problem with random data, i.e.,  we set $m=5000$. For the \textbf{StoManIAL}, we partition the 50,000 samples into 100 subsets, and in each iteration, we sample one subset. The tolerance is set $\mathrm{tol} = 10^{-8}\times n\times r$.   Figure \ref{fig:2} presents the results of the four algorithms for fixed $n=500$ and varying $r=1,2$ and $\mu=0.4,0.6,0.8$. The horizontal axis represents CPU time, while the vertical axis represents the objective function value gap: $F(x^k) - F_M$, where $F_M$ is given by the \textbf{ManIAL-I}. The results indicate that \textbf{StoManIAL} outperforms the deterministic version. Moreover, among the deterministic versions, \textbf{ManIAL-I} and \textbf{ManIAL-II} have similar performance. In conclusion, compared with the Rsub, our algorithms achieve better performance.

Next, we conduct experiments on two real datasets: \textit{coil100} \cite{nene1996columbia} and \textit{mnist} \cite{deng2012mnist}. The \textit{coil100} dataset contains $n = 7200$ RGB images of 100 objects taken from different angles. The $\textit{mnist}$ dataset has $n = 8,000$ grayscale digit images of size $28\times 28$. Our experiments, illustrated in Figures \ref{fig:3} and \ref{fig:4}, demonstrate that our algorithms outperform Rsub.

\begin{figure}[htpb]
\centering
\setlength{\abovecaptionskip}{0.cm}
\subfigure[$r = 1, \mu = 0.1$]{
\includegraphics[width=0.3\textwidth]{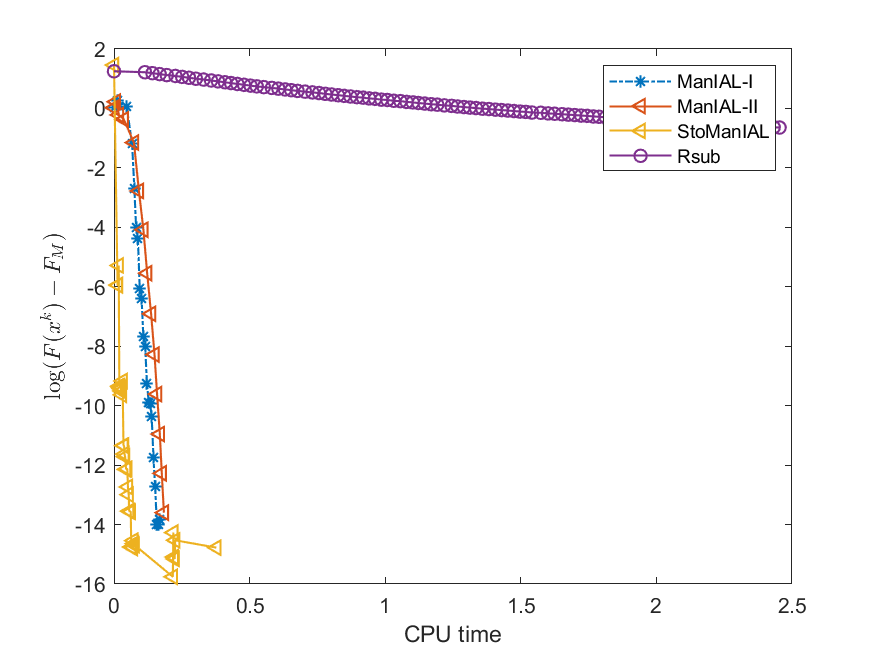}}
\subfigure[$r = 1, \mu = 0.2$]{
\includegraphics[width=0.3\textwidth]{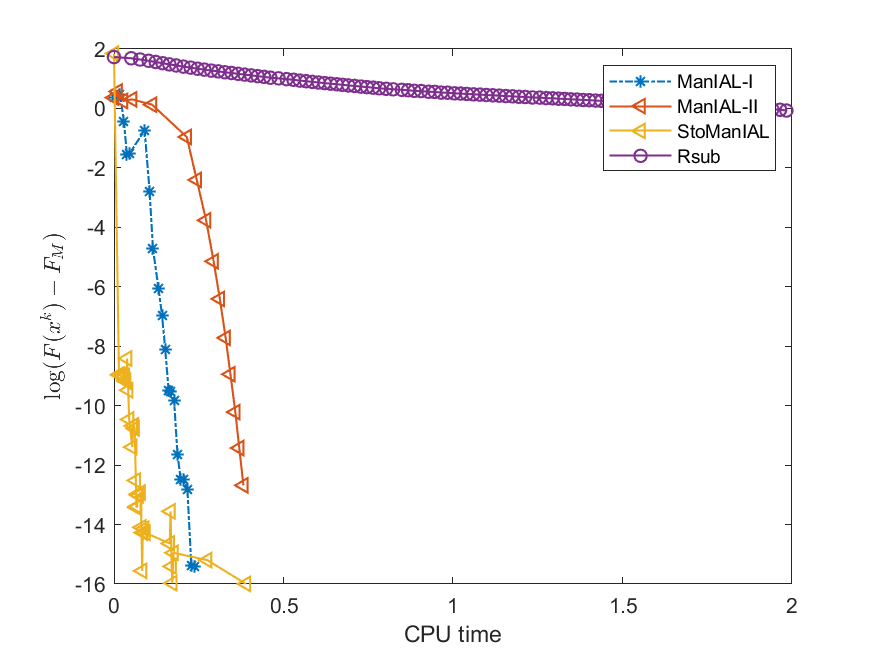}}
\subfigure[$r = 1, \mu = 0.3$]{
\includegraphics[width=0.3\textwidth]{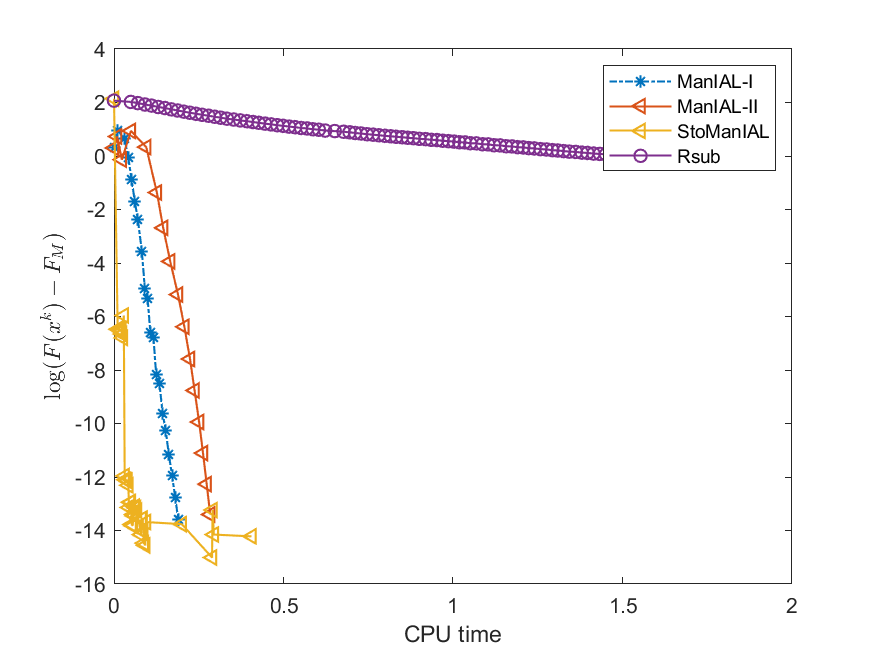}} \\
\subfigure[$r = 2, \mu = 0.1$]{
\includegraphics[width=0.3\textwidth]{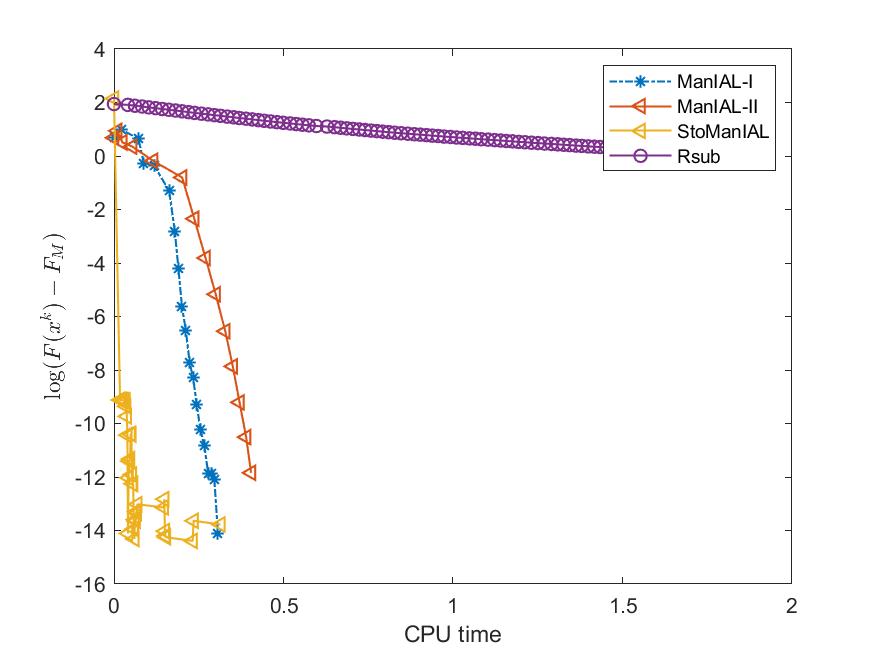}}
\subfigure[$r = 2, \mu = 0.2$]{
\includegraphics[width=0.3\textwidth]{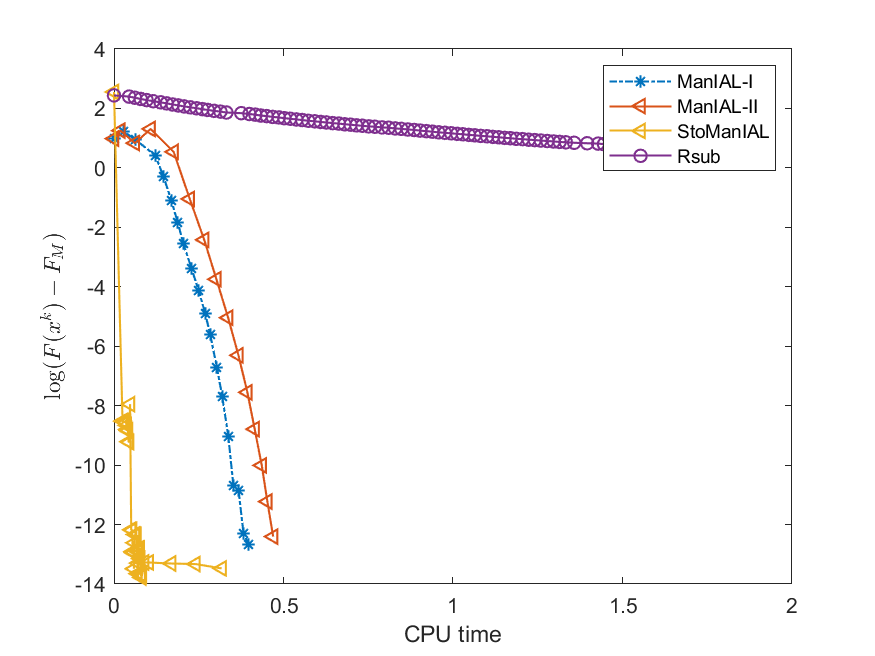}}
\subfigure[$r = 2, \mu = 0.3$]{
\includegraphics[width=0.3\textwidth]{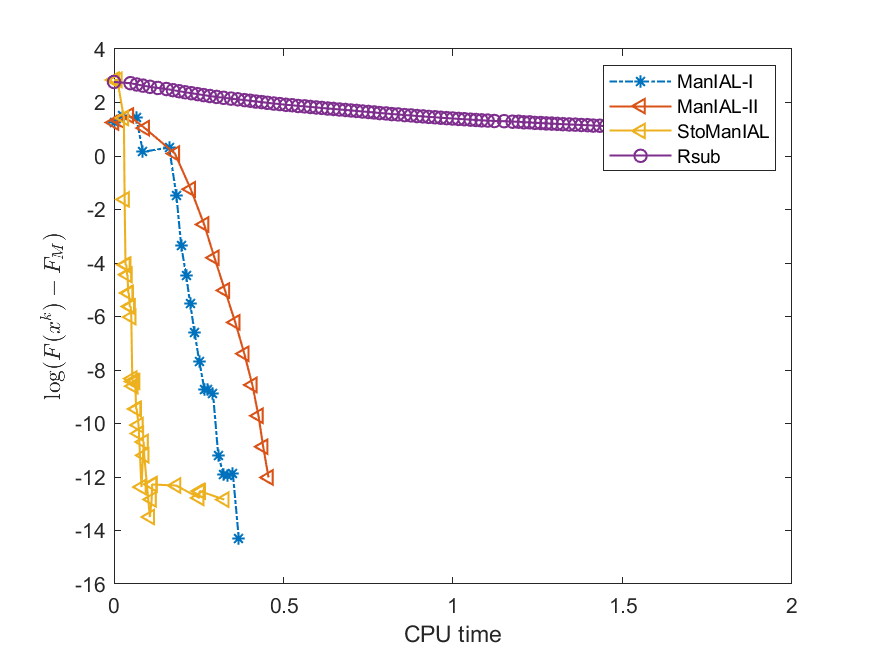}}
\caption{The performance of our algorithms on SPCA with \textit{mnist} data $(n = 784)$}
\label{fig:3}
\end{figure}

\begin{figure}[htpb]
\centering
\setlength{\abovecaptionskip}{0.cm}
\subfigure[$r = 1, \mu = 0.1$]{
\includegraphics[width=0.3\textwidth]{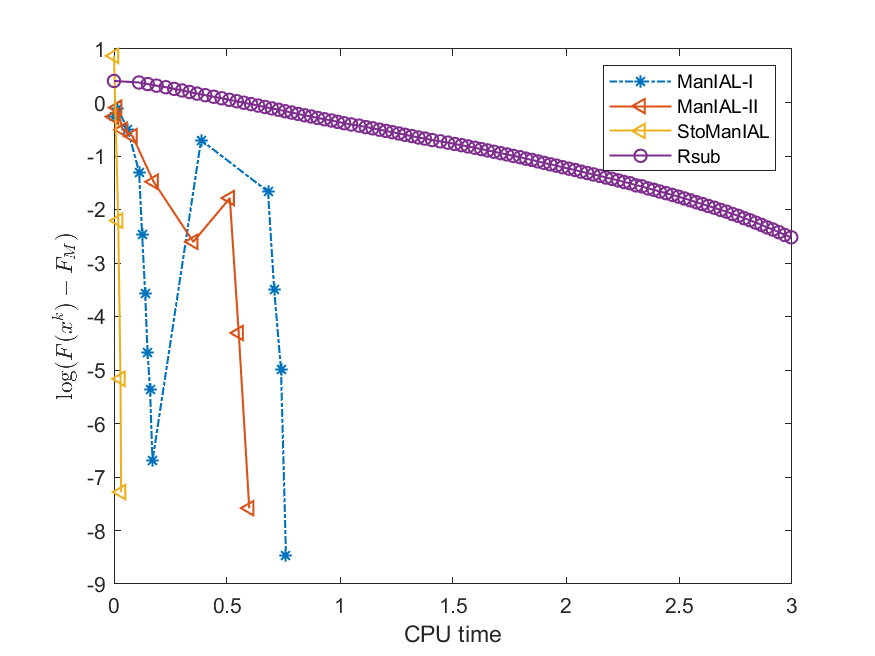}}
\subfigure[$r = 1, \mu = 0.2$]{
\includegraphics[width=0.3\textwidth]{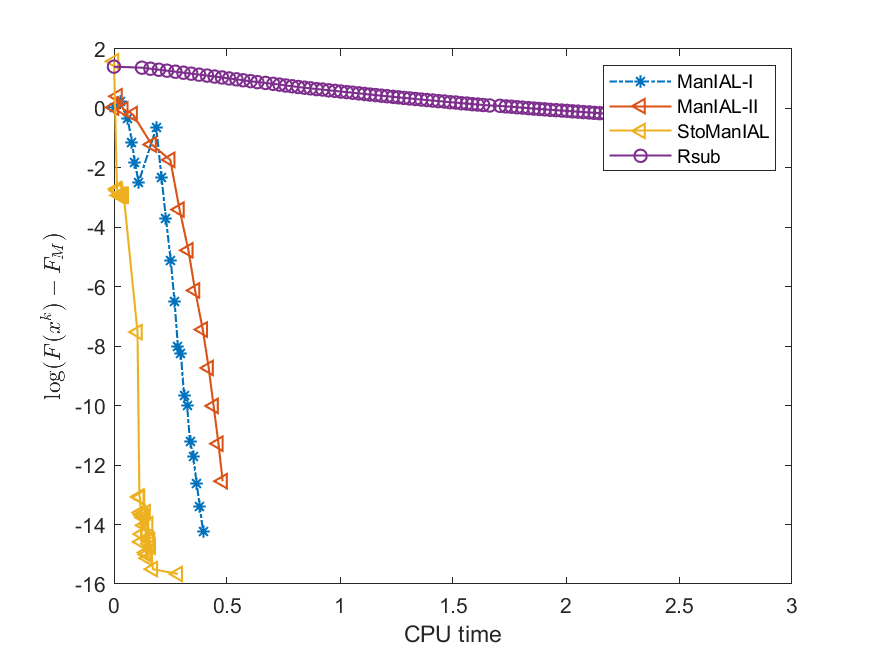}}
\subfigure[$r = 1, \mu = 0.3$]{
\includegraphics[width=0.3\textwidth]{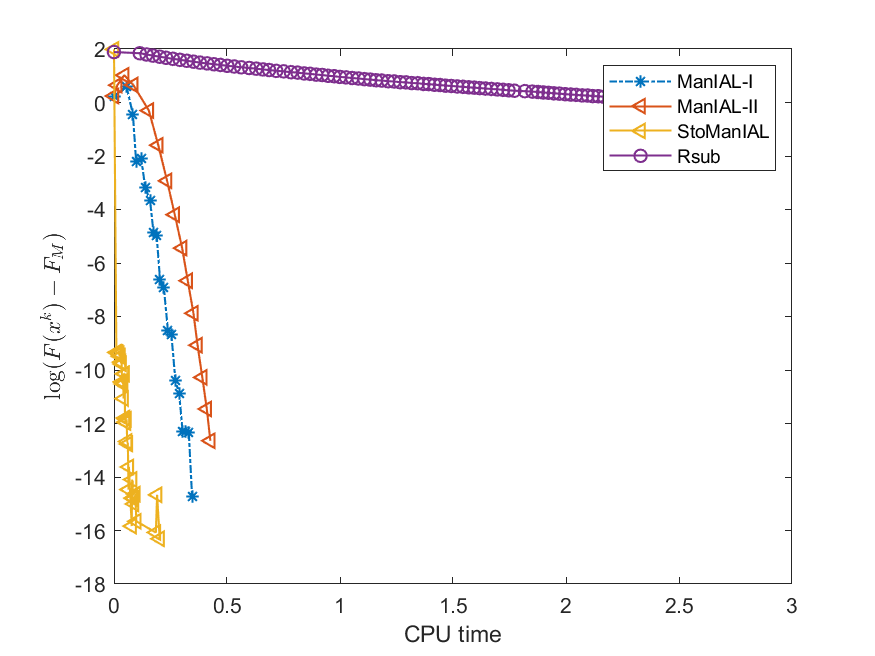}} \\
\subfigure[$r = 2, \mu = 0.1$]{
\includegraphics[width=0.3\textwidth]{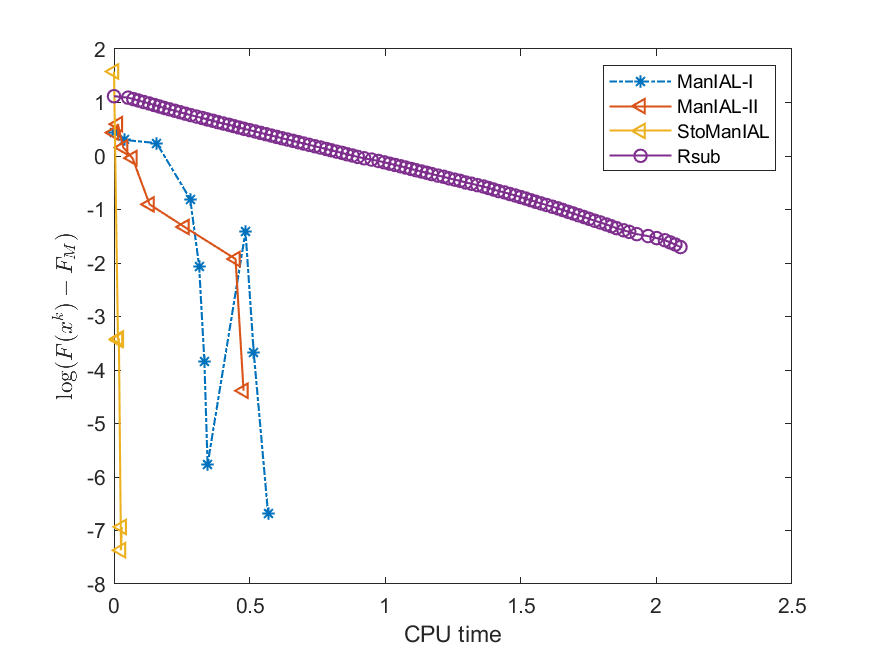}}
\subfigure[$r = 2, \mu = 0.2$]{
\includegraphics[width=0.3\textwidth]{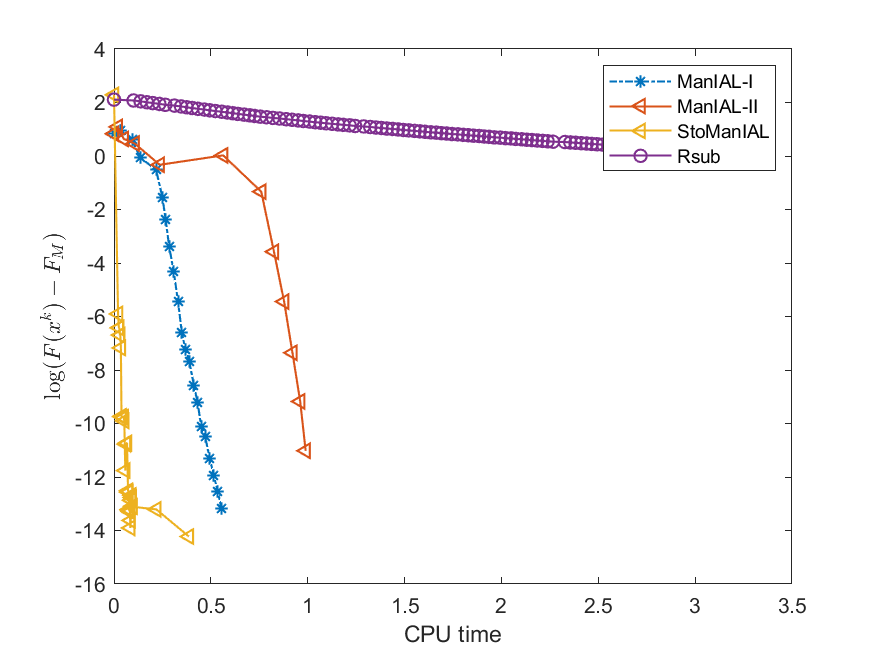}}
\subfigure[$r = 2, \mu = 0.3$]{
\includegraphics[width=0.3\textwidth]{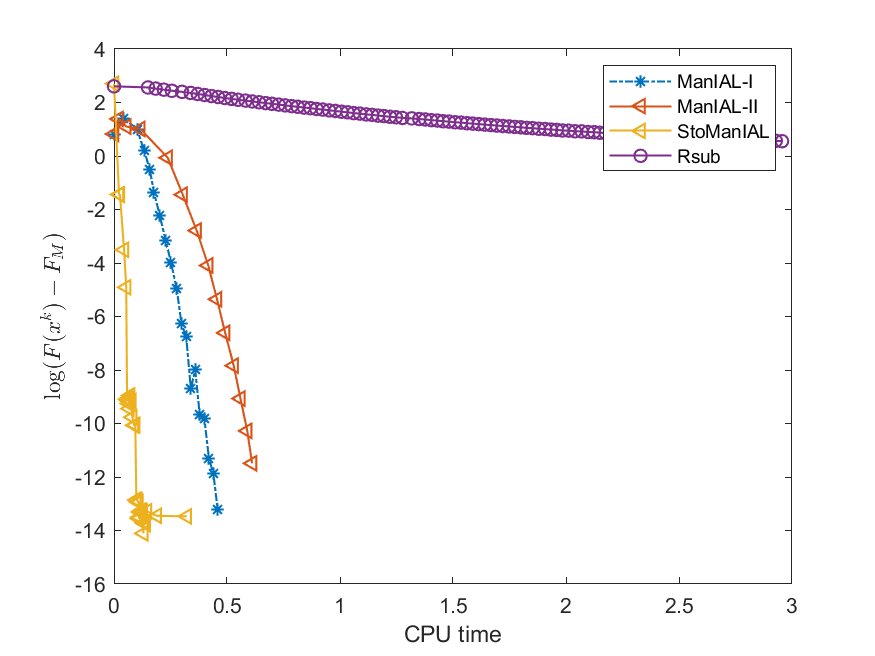}}
\caption{The performance of our algorithms on SPCA with \textit{coil100} data $(n = 1024)$}
\label{fig:4}
\end{figure}

\subsection{Sparse canonical correlation analysis}
Canonical correlation analysis (CCA) first proposed by Hotelling \cite{hotelling1992relations} aims to tackle the associations between two sets of variables.  The sparse CCA (SCCA) model is proposed to improve the interpretability of canonical variables by restricting the linear combinations to a subset of original variables. Suppose that there are two data sets: $X \in \mathbb{R}^{n \times p}$ containing $p$ variables and $Y \in$ $\mathbb{R}^{n \times q}$ containing $q$ variables, both are obtained from $n$ observations.  Let $\hat{\Sigma}_{x x}=\frac{1}{n} X^{\top} X, \quad \hat{\Sigma}_{y y}=\frac{1}{n} Y^{\top} Y$ be the sample covariance matrices of $X$ and $Y$, respectively, and $\hat{\Sigma}_{x y}=\frac{1}{n} X^{\top} Y$ be the sample cross-covariance matrix, then the SCCA can be formulated as
\begin{equation}\label{pro:scca}
    \min_{U\in \mathbb{R}^{p\times r}, V\in \mathbb{R}^{q\times r}} -\mathrm{tr}(U^\top \hat{\Sigma}_{x y} V) + \mu_1 \|U\|_1 + \mu_2 \|V\|_1,\; \mathrm{s.t.} \; U^\top \hat{\Sigma}_{x x} U = V^\top \hat{\Sigma}_{y y} V = I_r.
\end{equation}
Let us denote $\mathcal{M}_1 = \{U\in \mathbb{R}^{p\times r}~:~U^\top \hat{\Sigma}_{x x} U = I_r \}$ and $\mathcal{M}_2 = \{V\in \mathbb{R}^{p\times r}~:~V^\top \hat{\Sigma}_{y y} V = I_r \}$. Without loss of generality, we assume that $\hat{\Sigma}_{x x}$ and $\hat{\Sigma}_{yy}$ are positive definite (otherwise, we can impose a small identity matrix). In this case, $\mathcal{M}_1$ and $\mathcal{M}_2$ are two generalized Stiefel manifolds, problem \eqref{pro:scca} is a nonsmooth problem on a product manifold $\mathcal{M} := \mathcal{M}_1 \otimes \mathcal{M}_2$. Moreover, since $\hat{\Sigma}_{x y}=\frac{1}{n} X^{\top} Y = \frac{1}{n}\sum_{i}^n x_i^\top y_i$, where $x_i,y_i$ denote the $i$-th row of $X$ and $Y$, respectively, we can rewrite \eqref{pro:scca} as the following finite-sum form:
\begin{equation}\label{pro:scca-1}
    \min_{U\in \mathbb{R}^{p\times r}, V\in \mathbb{R}^{q\times r}} -\frac{1}{n}\sum_{i}^n\mathrm{tr}(U^\top x_i^\top y_i V) + \mu_1 \|U\|_1 + \mu_2 \|V\|_1,\; \mathrm{s.t.} \; U\in \mathcal{M}_1,V\in\mathcal{M}_2.
\end{equation}
Therefore, we can apply our algorithms \textbf{ManIAL} and \textbf{StoManIAL} to solve \eqref{pro:scca-1}. Owing to the poor performance in the SPCA problem, we only list the numerical comparisons of our algorithms. We compare the performance of the proposed algorithms on the SPCA problem with random data, i.e.,  we set $n=5000$. For the \textbf{StoManIAL}, we partition the 50,000 samples into 100 subsets, and in each iteration, we sample one subset. The tolerance is set $\mathrm{tol} = 10^{-8}\times p\times r$.   Figure \ref{fig:2} presents the results of the four algorithms for fixed $p,q=200$ and varying $r=1,2$ and $\mu=0.2,0.3,0.4$. The results show that all algorithms have similar performance.

\begin{figure}[htpb]
\centering
\setlength{\abovecaptionskip}{0.cm}
\subfigure[$r = 10, \mu = 0.4$]{
\includegraphics[width=0.3\textwidth]{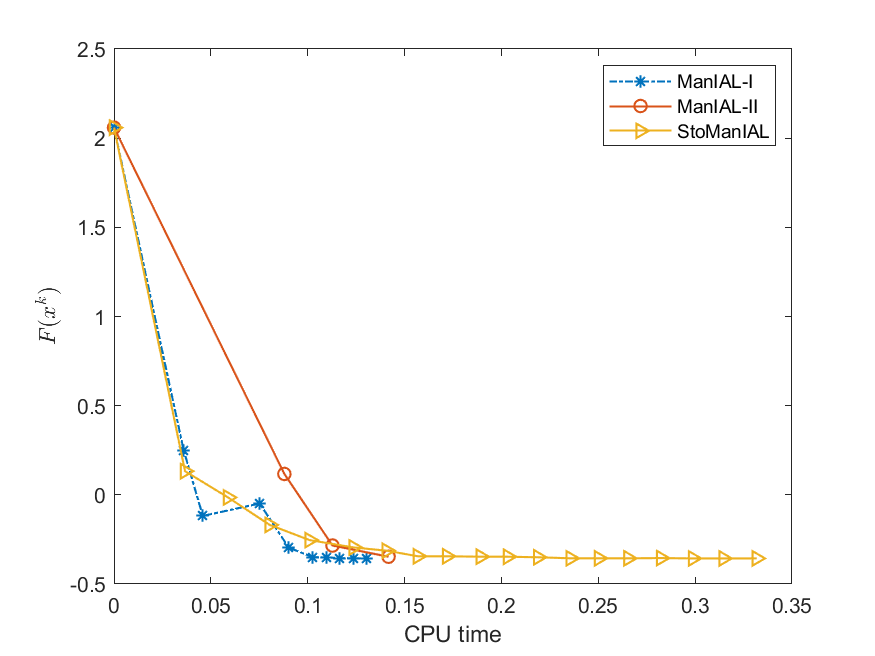}}
\subfigure[$r = 10, \mu = 0.6$]{
\includegraphics[width=0.3\textwidth]{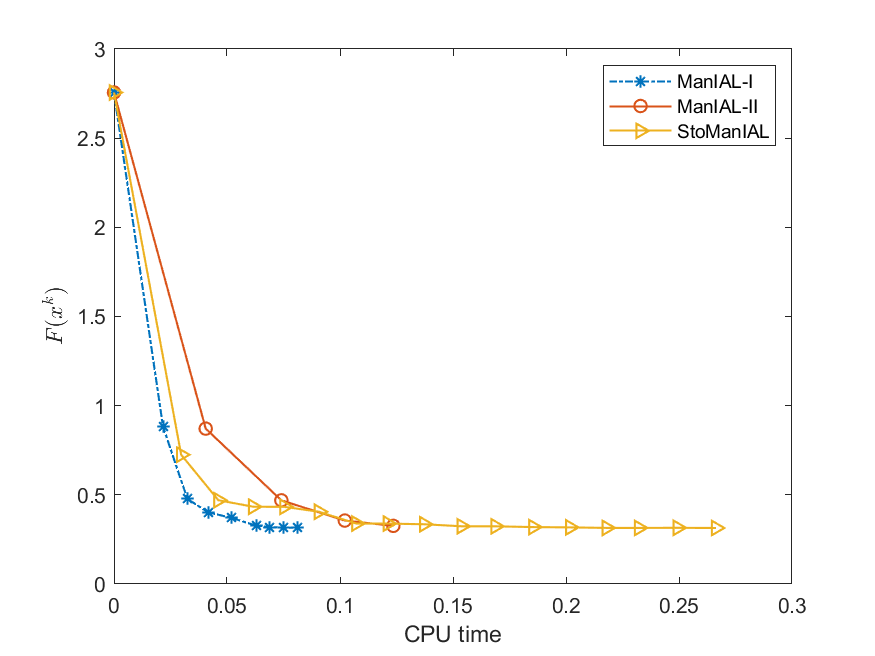}}
\subfigure[$r = 10, \mu = 0.8$]{
\includegraphics[width=0.3\textwidth]{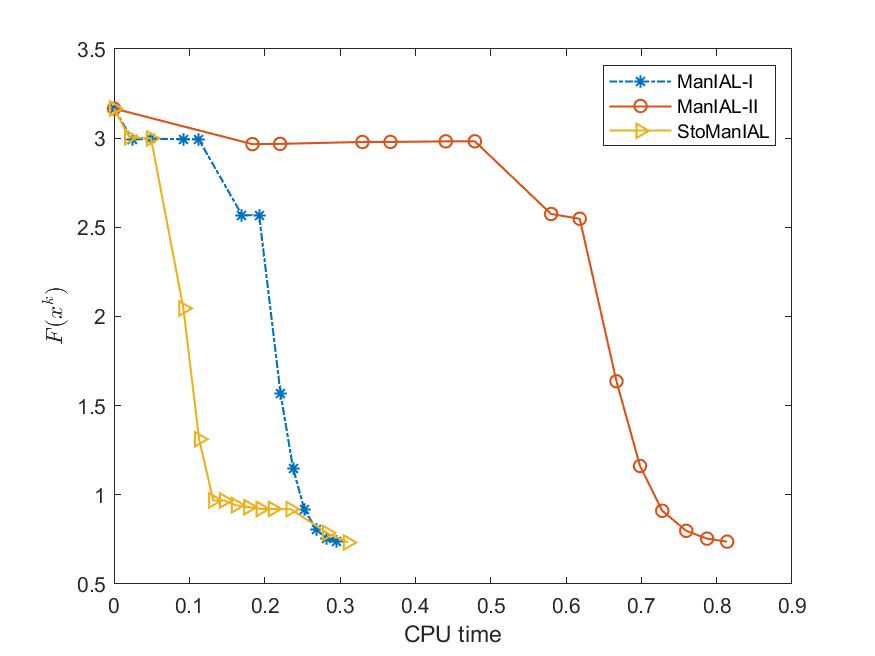}} \\
\subfigure[$r = 20, \mu = 0.4$]{
\includegraphics[width=0.3\textwidth]{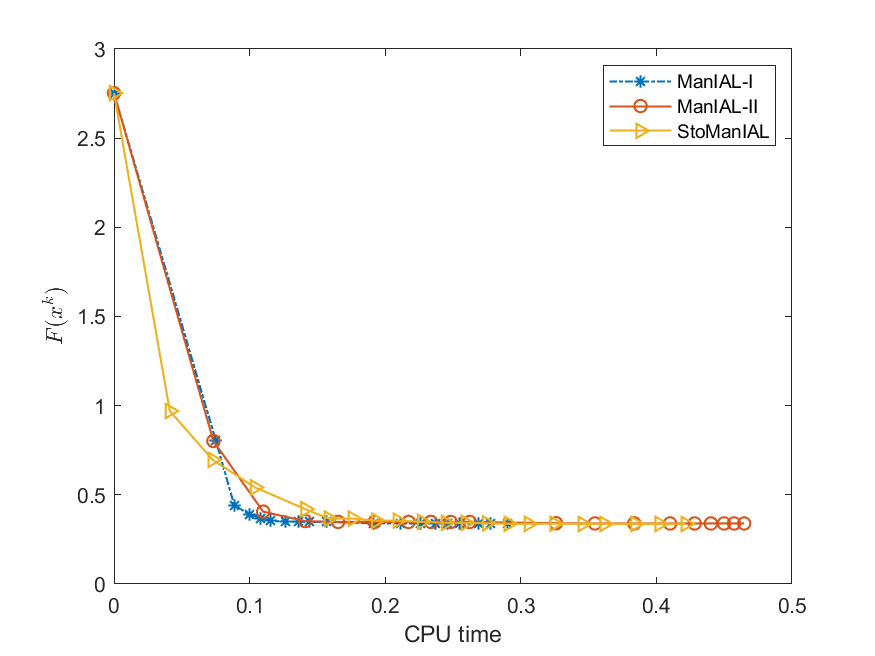}}
\subfigure[$r = 20, \mu = 0.6$]{
\includegraphics[width=0.3\textwidth]{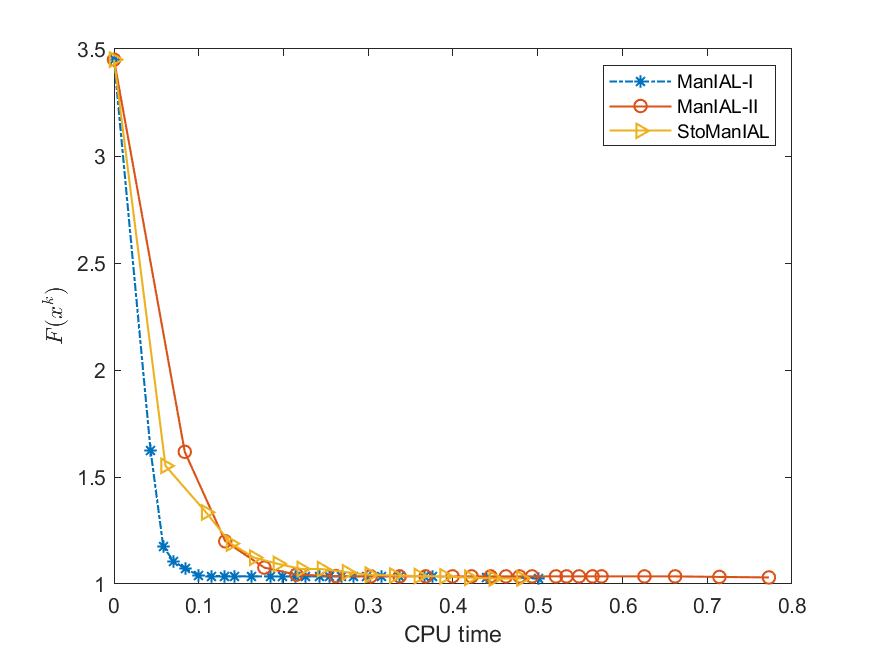}}
\subfigure[$r = 20, \mu = 0.8$]{
\includegraphics[width=0.3\textwidth]{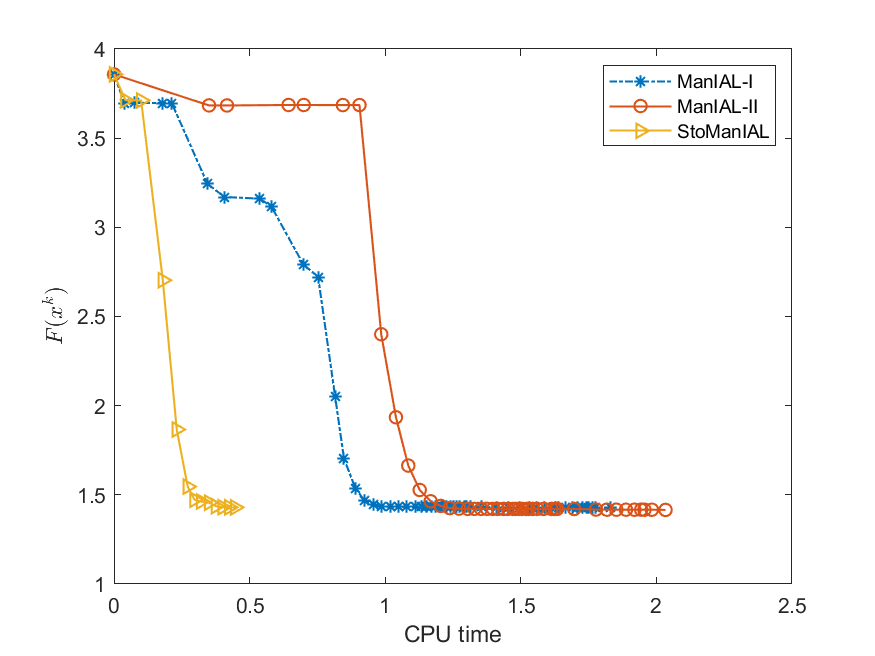}}
\caption{The performance of our algorithms on SCCA with random data $(n = 200)$}
\label{fig:5}
\end{figure}

\section{Conclusions}
This paper proposes two novel manifold inexact augmented Lagrangian methods, namely \textbf{ManIAL} and \textbf{StoManIAL}, with the oracle complexity guarantee. To the best of our knowledge, this is the first complexity result for the augmented Lagrangian method for solving nonsmooth problems on Riemannian manifolds. We prove that \textbf{ManIAL} and \textbf{StoManIAL} achieve the oracle complexity of $\mathcal{O}(\epsilon^{-3})$ and $\tilde{\mathcal{O}}(\epsilon^{-3.5})$, respectively.  Numerical experiments on SPCA and SCCA problems demonstrate that our proposed methods outperform an existing method with the previously best-known complexity result.

\bibliographystyle{siamplain}
\bibliography{optimization}

\newpage

\appendix

\section{Technical lemmas}
The next lemma is a technical observation and follows from Lemma 3 in \cite{cutkosky2019momentum}. 
\begin{lemma}\label{lemma:expect}
Suppose that the conditions \eqref{assum-rsrm-1}-\eqref{assum-rsrm-3} hold. Let us denote $s_{t} = d_t - \grad \psi(x_t)$. Then we have
$$
\begin{aligned}
& \mathbb{E}\left[\left(\grad \psi\left({x}_t, \xi_t\right)-\grad \psi\left({x}_t\right)\right) \cdot \eta_{t-1}^{-1}\left(1-a_t\right)^2 {s}_{t-1}\right]=0 \\
& \mathbb{E}\left[\left(\grad \psi\left({x}_t, \xi_t\right)-\grad \psi\left({x}_{t-1}, \xi_t\right)-\grad \psi\left({x}_t\right)+\grad \psi\left({x}_{t-1}\right)\right) \cdot \eta_{t-1}^{-1}\left(1-a_t\right)^2 {s}_{t-1}\right]=0.
\end{aligned}
$$
\end{lemma}

\begin{lemma}\label{lem:descent-lemma}
 Suppose that the conditions \eqref{assum-rsrm-1}-\eqref{assum-rsrm-3} hold.     If $\eta_t \leq \frac{1}{ L}$ in Algorithm \ref{alg:rsrm} for all $t$. Then
$$
\mathbb{E}\left[\psi\left(x_{t+1}\right)-\psi\left(x_t\right)\right] \leq \mathbb{E}\left[-\eta_t / 2\left\|\grad \psi\left(x_t\right)\right\|^2+ \eta_t / 2\left\|d_t-\grad  \psi\left(x_t\right)\right\|^2\right].
$$
\end{lemma} 

\begin{proof}
    By retr-smoothness of $\psi$, we have
$$
\begin{aligned}
\psi\left(x_{t+1}\right) & \leq \psi\left(x_t\right)-\left\langle\grad  \psi\left(x_t\right), \eta_t d_t\right\rangle+\frac{\eta_t^2 L}{2}\left\|d_t\right\|^2 \\
& =\psi\left(x_t\right)-\frac{\eta_t}{2}\left\|\grad  \psi\left(x_t\right)\right\|^2-\frac{\eta_t}{2}\left\|d_t\right\|^2+\frac{\eta_t}{2}\left\|d_t-\grad  \psi\left(x_t\right)\right\|^2+\frac{L \eta_t^2}{2}\left\|d_t\right\|^2 \\
& =\psi\left(x_t\right)-\frac{\eta_t}{2}\left\|\grad  \psi\left(x_t\right)\right\|^2+\frac{\eta_t}{2}\left\|d_t-\grad  \psi\left(x_t\right)\right\|^2-\left(\frac{\eta_t}{2}-\frac{L \eta_t^2}{2}\right)\left\|d_t\right\|^2 \\
& \leq \psi\left(x_t\right)-\frac{\eta_t}{2}\left\|\grad  \psi\left(x_t\right)\right\|^2+\frac{\eta_t}{2}\left\|d_t-\grad  \psi\left(x_t\right)\right\|^2,
\end{aligned}
$$
where for the last inequality we choose $\eta_t \leq \frac{1}{L}$.
\end{proof} 

The following technical lemma, which follows from Lemma 3 in \cite{cutkosky2019momentum}, provides a recurrence variance bound.
\begin{lemma}\label{lem:bound-st}
  Suppose that the conditions \eqref{assum-rsrm-1}-\eqref{assum-rsrm-3} hold.   With the notation in Algorithm \ref{alg:rsrm}, let $s_t=d_t-\grad  \psi\left(x_t\right)$,  we have
$$
\begin{aligned}
\mathbb{E}\left[\|s_t\|^2 \right] \leq \mathbb{E}\left[2a_t^2 G_t^2 \right. & + (1-a_t)^2 (1+4L^2 \eta_{t-1}^2) \|s_{t-1}\|^2  \\
& + \left. 4(1-a_t)^2L^2 \eta_{t-1}^2 \|\grad \psi(x_{t-1}) \|^2\right].
\end{aligned}
$$
\end{lemma}

\begin{proof}

By definition of $s_t$ and the notation in Algorithm \ref{alg:rsrm}, we have $s_t=d_t-\grad  \psi\left(x_t\right)=\grad  \psi\left(x_t, \xi_t\right)+(1-$ $\left.a_t\right) \mathcal{T}_{x_{t-1}}^{x_t}\left(d_{t-1}-\grad  \psi \left(x_{t-1}, \xi_t\right)\right)-\grad  \psi\left(x_t\right)$. Hence, we can write
$$
\begin{aligned}
& \mathbb{E} {\left[\left\|s_t\right\|^2\right] =\mathbb{E}\left[\left\|\grad  \psi \left(x_t, \xi_t\right)+\left(1-a_t\right)\mathcal{T}_{x_{t-1}}^{x_t}\left(d_{t-1}-\grad  \psi \left(x_{t-1}, \xi_t\right)\right)-\grad  \psi\left(x_t\right)\right\|^2\right] } \\
&= \mathbb{E}\left[ \| a_t(\grad  \psi(x_t, \xi_t)-\grad  \psi(x_t))+(1-a_t)(\grad  \psi(x_t, \xi_t)- \mathcal{T}_{x_{t-1}}^{x_t}\grad  \psi(x_{t-1}, \xi_t)) \right.  \\ 
&  +(1-a_t)(\mathcal{T}_{x_{t-1}}^{x_t}\grad  \psi(x_{t-1}) - \grad  \psi(x_t))+ \left. \left(1-a_t\right)\mathcal{T}_{x_{t-1}}^{x_t}\left(d_{t-1}- \grad  \psi\left(x_{t-1}\right)\right) \|^2\right] \\
& \leq \mathbb{E}\left[2a_t^2  \left\|\grad  \psi\left(x_t, \xi_t\right)-\grad  \psi\left(x_t\right)\right\|^2 + \left(1-a_t\right)^2\left\|s_{t-1}\right\|^2 + 2(1-a_t)^2 \right. \\
&\left.  \left\| \grad  \psi \left(x_t, \xi_t\right)-\mathcal{T}_{x_{t-1}}^{x_t}\grad  \psi \left(x_{t-1}, \xi_t\right) - \grad  \psi\left(x_t\right) +\mathcal{T}_{x_{t-1}}^{x_t} \grad  \psi\left(x_{t-1}\right)\right\|^2 \right] \\
& \leq \mathbb{E}\left[ 2a_t^2  \| \grad  \psi \left(x_t, \xi_t\right)\|^2 + (1-a_t)^2\|s_{t-1}\|^2 \right.  \\
& \left.+2(1-a_t)^2 \| \grad  \psi \left(x_t, \xi_t\right)-\mathcal{T}_{x_{t-1}}^{x_t}\grad  \psi \left(x_{t-1}, \xi_t\right) \|^2   \right] \\
& \leq \mathbb{E}\left[ 2a_t^2 G_t^2 + (1-a_t)^2  \|s_{t-1}\|^2 + 2(1-a_t)^2L^2 \eta_{t-1}^2\| d_{t-1} \|^2   \right] \\
& \leq \mathbb{E}\left[ 2a_t^2 G_t^2 + (1-a_t)^2  \|s_{t-1}\|^2 + 4(1-a_t)^2L^2 \eta_{t-1}^2 (\| s_{t-1} \|^2 + \|\grad \psi(x_{t-1}) \|^2)   \right] \\
& \leq \mathbb{E}\left[ 2a_t^2 G_t^2 + (1-a_t)^2 (1+4L^2 \eta_{t-1}^2) \|s_{t-1}\|^2 + 4(1-a_t)^2L^2 \eta_{t-1}^2 \|\grad \psi(x_{t-1}) \|^2\right], 
\end{aligned}
$$
where the first inequality uses Lemma \ref{lemma:expect} and  $\|x+y\|^2 \leq 2\|x\|^2+2\|y\|^2$, the second inequality follows from $\mathbb{E}\|x - \mathbb{E}[x]\|^2 \leq \mathbb{E}\|x\|^2$, the third inequality uses the update rule of $x_{t}$ and \eqref{lipschi-rie}.  
\end{proof}

\section{The proof of Lemma \ref{lem:rsrm}}
\begin{proof}
First, since $w \geq(4 L \kappa)^3$, we have from the update formula of $\eta_t$ that $\eta_t \leq \frac{\kappa}{w^{1/3}} \leq \frac{1}{4 L}<\frac{1}{L}$. 
 It follows from Lemma \ref{lem:descent-lemma} that
\begin{equation}\label{eq:descent}
    \mathbb{E}\left[\psi\left(x_{t+1}\right)-\psi\left(x_t\right)\right] \leq \mathbb{E}\left[-\eta_t / 2\left\|\grad \psi\left(x_t\right)\right\|^2+ \eta_t / 2\left\|d_t-\grad  \psi\left(x_t\right)\right\|^2\right].
\end{equation}
Consider the potential function $\Phi_t=\psi\left(x_t\right)+\frac{1}{24 L^2 \eta_{t-1}}\left\|s_t\right\|^2$. Since $a_{t+1}=c \eta_t^2$ and $w \geq \left(\frac{c\kappa}{4L}\right)^3$, we have that for all $t$
$$a_{t+1} \leq \frac{ c \eta_t \kappa}{w^{1/3}} \leq  \frac{c \kappa}{4 L w^{1 / 3}} \leq 1.$$
Then, we first consider $\eta_t^{-1}\left\|s_{t+1}\right\|^2-\eta_{t-1}^{-1}\left\|s_t\right\|^2$. Using Lemma \ref{lem:bound-st}, we obtain
$$
\begin{aligned}
& \mathbb{E}  {\left[\eta_t^{-1}\left\|s_{t+1}\right\|^2-\eta_{t-1}^{-1}\left\|s_t\right\|^2\right] } \\
 \leq & \mathbb{E}\left[2 \frac{a_{t+1}^2}{\eta_t} G_{t+1}^2+\frac{\left(1-a_{t+1}\right)^2\left(1+4 L^2 \eta_t^2\right)\left\|s_t\right\|^2}{\eta_t}+4\left(1-a_{t+1}\right)^2 L^2 \eta_t\left\|\grad  \psi\left(x_t\right)\right\|^2 \right.\\
 & \left. -\frac{\left\|s_t\right\|^2}{\eta_{t-1}}\right] \\
 \leq & \mathbb{E}[\underbrace{2\frac{a_{t+1}^2}{\eta_t} G_{t+1}^2}_{A_t}+\underbrace{\left(\eta_t^{-1}\left(1-a_{t+1}\right)\left(1+4 L^2 \eta_t^2\right)-\eta_{t-1}^{-1}\right)\left\|s_t\right\|^2}_{B_t}+\underbrace{4 L^2 \eta_t\left\|\grad  \psi\left(x_t\right)\right\|^2}_{C_t}] .
\end{aligned}
$$

Let us focus on the terms of this expression individually. For the first term, $A_t$, observing that $w \geq 2 G^2 \geq$ $G^2+G_{t+1}^2$, we have
$$
\begin{aligned}
\sum_{t=1}^T A_t&=\sum_{t=1}^T 2 c^2 \eta_t^3 G_{t+1}^2  =\sum_{t=1}^T \frac{2 \kappa^3 c^2 G_{t+1}^2}{w+\sum_{i=1}^t G_i^2} \\
& \leq \sum_{t=1}^T \frac{2 \kappa^3 c^2 G_{t+1}^2}{G^2+\sum_{i=1}^{t+1} G_i^2}  \leq 2 \kappa^3 c^2 \ln \left(1+\sum_{t=1}^{T+1} \frac{G_t^2}{G^2}\right)  \\ & \leq 2 \kappa^3 c^2 \ln (T+2),
\end{aligned}
$$
where in the second to last inequality we used Lemma 4 in \cite{cutkosky2019momentum}.
For the second term $B_t$, we have
$$
B_t \leq \left(\eta_t^{-1}-\eta_{t-1}^{-1}+\eta_t^{-1}\left(4 L^2 \eta_t^2-a_{t+1}\right)\right)\left\|s_t\right\|^2\leq \left(\eta_t^{-1}-\eta_{t-1}^{-1}+\eta_t\left(4 L^2-c\right)\right)\left\|s_t\right\|^2.
$$

Let us estimate $\frac{1}{\eta_t}-\frac{1}{\eta_{t-1}}$. Using the concavity of $x^{1 / 3}$, we have $(x+y)^{1 / 3} \leq x^{1 / 3}+y x^{-2 / 3} / 3$. Therefore:
$$
\begin{aligned}
\frac{1}{\eta_t}-\frac{1}{\eta_{t-1}} & =\frac{1}{\kappa}\left[\left(w+\sum_{i=1}^t G_i^2\right)^{1 / 3}-\left(w+\sum_{i=1}^{t-1} G_i^2\right)^{1 / 3}\right] \leq \frac{G_t^2}{3 \kappa \left(w+\sum_{i=1}^{t-1} G_i^2\right)^{2 / 3}} \\
& \leq \frac{G_t^2}{3 \kappa\left(w-G^2+\sum_{i=1}^t G_i^2\right)^{2 / 3}} \leq \frac{G_t^2}{3 \kappa \left(w / 2+\sum_{i=1}^t G_i^2\right)^{2 / 3}} \\
& \leq \frac{2^{2 / 3} G_t^2}{3 \kappa \left(w+\sum_{i=1}^t G_i^2\right)^{2 / 3}} \leq \frac{2^{2 / 3} G^2}{3 \kappa^3} \eta_t^2 \leq \frac{2^{2 / 3} G^2}{12 L \kappa^3} \eta_t \leq \frac{G^2}{7 L \kappa^3} \eta_t,
\end{aligned}
$$
where we use $\eta_t \leq \frac{1}{4 L}$.
Further, since $c=10 L^2+G^2 /\left(7 L \kappa^3\right)$, we have
$$
\eta_t\left(4 L^2-c\right) \leq-6L^2 \eta_t-G^2 \eta_t /\left(7 L \kappa^3\right) .
$$
Thus, we obtain $B_t \leq- 6L^2 \eta_t\left\|s_t\right\|^2$. Putting all this together yields:
\be\label{eq:sum-st+1-st}
\frac{1}{12 L^2} \sum_{t=1}^T\left(\frac{\left\|s_{t+1}\right\|^2}{\eta_t}-\frac{\left\|s_t\right\|^2}{\eta_{t-1}}\right) \leq \frac{\kappa^3 c^2}{6 L^2} \ln (T+2)+\sum_{t=1}^T\left[\frac{\eta_t}{3}\left\|\grad  \psi\left(x_t\right)\right\|^2-\frac{ \eta_t}{2}\left\|s_t\right\|^2\right].
\ee
Now, we are ready to analyze the potential $\Phi_t$. It follows from \eqref{eq:descent} that
$$
\mathbb{E}\left[\Phi_{t+1}-\Phi_t\right] \leq \mathbb{E}\left[-\frac{\eta_t}{2}\left\|\grad  \psi\left(x_t\right)\right\|^2+\frac{ \eta_t}{2}\left\|s_t\right\|^2+\frac{1}{12 L^2} \left( \frac{\|s_{t+1} \|^2}{\eta_t}-\frac{\left\|s_t\right\|^2}{\eta_{t-1}} \right) \right].
$$
Summing over $t$ and using \eqref{eq:sum-st+1-st}, we obtain
$$
\begin{aligned}
\mathbb{E}\left[\Phi_{T+1}-\Phi_1\right] & \leq \sum_{t=1}^T \mathbb{E}\left[-\frac{\eta_t}{2}\left\|\grad  \psi\left(x_t\right)\right\|^2+\frac{ \eta_t}{2}\left\|s_t\right\|^2+\frac{1}{12 L^2}\left( \frac{\|s_{t+1} \|^2}{\eta_t}-\frac{\left\|s_t\right\|^2}{\eta_{t-1}} \right)\right] \\
& \leq \mathbb{E}\left[\frac{\kappa^3 c^2}{6 L^2} \ln (T+2)-\sum_{t=1}^T \frac{\eta_t}{6}\left\|\grad  \psi\left(x_t\right)\right\|^2\right].
\end{aligned}
$$
Reordering the terms, we have
$$
\begin{aligned}
\mathbb{E}\left[\sum_{t=1}^T \eta_t\left\|\grad  \psi\left(x_t\right)\right\|^2\right] & \leq \mathbb{E}\left[6\left(\Phi_1-\Phi_{T+1}\right)+ \frac{\kappa^3 c^2}{2 L^2} \ln (T+2)\right] \\
& \leq 6\left(\psi\left(x_1\right)-F^{\star}\right)+\frac{\mathbb{E}\left[\left\|s_1\right\|^2\right]}{2 L^2 \eta_0}+\frac{\kappa^3 c^2}{L^2} \ln (T+2) \\
& \leq 6\left(\psi\left(x_1\right)-F^{\star}\right)+ \frac{w^{1 / 3} \delta^2}{2 L^2 \kappa} +\frac{\kappa^3 c^2}{L^2} \ln (T+2),
\end{aligned}
$$
where the last inequality is given by the definition of $d_1$ and $\eta_0$ in the algorithm.
Now, we relate $\mathbb{E}\left[\sum_{t=1}^T \eta_t\left\|\grad  \psi\left(x_t\right)\right\|^2\right]$ to $\mathbb{E}\left[\sum_{t=1}^T\left\|\grad  \psi\left(x_t\right)\right\|^2\right]$. First, since $\eta_t$ is decreasing,
$$
\mathbb{E}\left[\sum_{t=1}^T \eta_t\left\|\grad  \psi\left(x_t\right)\right\|^2\right] \geq \mathbb{E}\left[\eta_T \sum_{t=1}^T\left\|\grad  \psi\left(x_t\right)\right\|^2\right]
$$
Therefore, if we set $M=\frac{1}{\kappa}\left[6\left(\psi\left(x_1\right)-\psi_{\min}\right)+\frac{w^{1 / 3} \delta^2}{2 L^2 \kappa}+\frac{\kappa^3 c^2}{ L^2} \ln (T+2)\right]$, to get
\begin{equation}\label{eq:sum-grad-bound}
\begin{aligned}
\mathbb{E}\left[\sqrt{\sum_{t=1}^T\left\|\grad  \psi\left(x_t\right)\right\|^2}\right]^2 & \leq \mathbb{E}\left[\frac{6\left(\psi\left(x_1\right)-F^{\star}\right)+\frac{w^{1 / 3} \delta^2}{2 L^2 \kappa}+\frac{\kappa^3 c^2}{ L^2} \ln (T+2)}{\eta_T}\right] \\
& =\mathbb{E}\left[\frac{\kappa M}{\eta_T}\right] \leq \mathbb{E}\left[M\left(w+\sum_{t=1}^T G_t^2\right)^{1 / 3}\right].
\end{aligned}
\end{equation}
It follows from \eqref{assum-rsrm-1} that
\begin{equation}
    G_t^2=\left\|\grad  \psi\left(x_t\right)+\grad  \psi \left(x_t, \xi_t\right)-\grad  \psi\left(x_t\right)\right\|^2 \leq 2\left\|\grad  \psi\left(x_t\right)\right\|^2+2\delta^2.
\end{equation}
Plugging this in \eqref{eq:sum-grad-bound} and using $(a+b)^{1 / 3} \leq a^{1 / 3}+b^{1 / 3}$ we obtain:
$$
\begin{aligned}
& \mathbb{E}\left[\sqrt{\sum_{t=1}^T\left\|\grad  \psi\left(x_t\right)\right\|^2}\right]^2 \\
\leq & \mathbb{E}\left[M\left(w+2 \sum_{t=1}^T\delta^2\right)^{1 / 3}+M \left(2\sum_{t=1}^T\left\|\grad  \psi\left(x_t\right)\right\|^2\right)^{1 / 3}\right] \\
\leq & M\left(w+2 T \delta^2\right)^{1 / 3}+\mathbb{E}\left[2^{1 / 3} M\left(\sqrt{\sum_{t=1}^T\left\|\grad  \psi\left(x_t\right)\right\|^2}\right)^{2 / 3}\right] \\
\leq & M\left(w+2 T \delta^2\right)^{1 / 3}+2^{1 / 3} M\left(\mathbb{E}\left[\sqrt{\sum_{t=1}^T\left\|\grad  \psi\left(x_t\right)\right\|^2}\right]\right)^{2 / 3},
\end{aligned}
$$
where we have used the concavity of $x \mapsto x^a$ for all $a \leq 1$ to move expectations inside the exponents. Now, define $X=\sqrt{\sum_{t=1}^T\left\|\grad  \psi\left(x_t\right)\right\|^2}$. Then the above can be rewritten as:
$$
(\mathbb{E}[X])^2 \leq M\left(w+2 T \delta^2\right)^{1 / 3}+2^{1 / 3} M(\mathbb{E}[X])^{2 / 3}.
$$
Note that this implies that either $(\mathbb{E}[X])^2 \leq 2 M\left(w+T \delta^2\right)^{1 / 3}$, or $(\mathbb{E}[X])^2 \leq 2 \cdot 2^{1 / 3} M(\mathbb{E}[X])^{2 / 3}$. Solving for $\mathbb{E}[X]$ in these two cases, we obtain
$$
\mathbb{E}[X] \leq \sqrt{2 M}\left(w+2 T \delta^2\right)^{1 / 6}+2 M^{3 / 4}.
$$
Finally, observe that by Cauchy-Schwarz we have $\sum_{t=1}^T\left\|\grad  \psi\left(x_t\right)\right\| / T \leq X / \sqrt{T}$ so that
$$
\begin{aligned}
\mathbb{E}\left[\sum_{t=1}^T \frac{\left\|\grad  \psi\left(x_t\right)\right\|}{T}\right] & \leq \frac{\sqrt{2 M}\left(w+2 T \delta^2\right)^{1 / 6}+2 M^{3 / 4}}{\sqrt{T}} \\
&\leq \frac{w^{1 / 6} \sqrt{2 M}+2 M^{3 / 4}}{\sqrt{T}}+\frac{2 \sigma^{1 / 3}}{T^{1 / 3}}, 
\end{aligned}
$$
where we used $(a+b)^{1 / 3} \leq a^{1 / 3}+b^{1 / 3}$ in the last inequality.

\end{proof}

\end{document}